\documentclass[12pt]{imsart}
\RequirePackage[OT1]{fontenc}
\RequirePackage{amsthm,amsmath,amssymb}
\RequirePackage{natbib}
\usepackage{hyperref}
\usepackage{mathtools}
\usepackage{hypernat}
\usepackage{marvosym}
\usepackage[hang,small,bf]{caption}

\usepackage{amssymb,tabularx,multicol,multirow,booktabs}
\usepackage{mhequ}
\usepackage{relsize}

\newtheorem{theorem}{Theorem}[section]
\newtheorem{proposition}[theorem]{Proposition}
\newtheorem{lemma}[theorem]{Lemma}
\newtheorem{corollary}[theorem]{Corollary}
\newtheorem{remark}{Remark}[section]

\DeclareMathOperator*{\argmin}{arg\,min}
\startlocaldefs
\endlocaldefs

\newcommand{\be}{\boldsymbol{\varepsilon}}

\newcommand{\E}{\mathbb{E}}

\newcommand{\s}{{\mathcal{S}}}

\newcommand{\R}{\mathbb{R}}

\newcommand{\D}{\mathbb{D}}

\newcommand{\G}{{\mathbb{G}}}
\newcommand{\F}{{\mathbb F}}

\renewcommand{\P}{{\mathbb P}}

\newcommand{\B}{\mathbb{B}}
\newcommand{\M}{\mathcal{M}}

\newcommand{\C}{\mathcal{C}}

\def \be{\begin{equs}}
\def \ee{\end{equs}}
\def \fhat {\hat{f_n}}

\begin{document}

\begin{frontmatter}
\title{On Efficiency of the Plug-in Principle for Estimating  Smooth Integrated Functionals of a Nonincreasing Density}
\runtitle{On Efficiency of the Grenander based Plug-in Estimator}

\begin{aug}
\author{\fnms{Rajarshi} \snm{Mukherjee}
\ead[label=e1]{rmukherj@hsph.harvard.edu}}
\and
\author{\fnms{Bodhisattva} \snm{Sen}\thanksref{t2}\ead[label=e2]{bodhi@stat.columbia.edu}}


\thankstext{t2}{Supported by NSF Grants DMS-17-12822 and AST-16-14743.}
\runauthor{R.~Mukherjee and B.~Sen}

\affiliation{Harvard University and Columbia University}

\address{Department of Biostatistics \\ Harvard University \\ 655 Huntington Avenue \\ Boston, MA 02115 \\ \printead{e1}}

\address{Department of Statistics \\ Columbia University \\ 1255 Amsterdam Avenue \\ New York, NY 10027 \\ \printead{e2}}
\end{aug}

\begin{abstract}
We consider the problem of estimating smooth integrated functionals of a monotone nonincreasing density $f$ on $[0,\infty)$ using the nonparametric maximum likelihood based plug-in estimator. We find the exact asymptotic distribution of this natural (tuning parameter-free) plug-in estimator, properly normalized. In particular, we show that the simple plug-in estimator is always $\sqrt{n}$-consistent, and is additionally asymptotically normal with zero mean  and the semiparametric efficient variance for estimating a subclass of integrated functionals. Compared to the previous results on this topic (see e.g.,~\cite{nickl2008uniform},~\cite{H14},~\citet[Chapter 7]{gine2015mathematical},~and \cite{sohl2015uniform}) our results hold  for a much larger class of functionals (which include linear and non-linear functionals) under less restrictive assumptions on the underlying $f$ --- we do not require $f$ to be (i) smooth, (ii) bounded away from $0$, or (iii) compactly supported. Further, when $f$ is the uniform distribution on a compact interval we explicitly characterize the asymptotic distribution of the plug-in estimator  --- which now converges at a non-standard rate --- thereby extending the results in~\cite{GP83} for the case of the quadratic functional.

\end{abstract}

\begin{keyword}
\kwd{efficient estimation}
\kwd{empirical distribution function}
\kwd{Hadamard directionally differentiable}
\kwd{least concave majorant}
\kwd{nonparametric maximum likelihood estimator}
\kwd{shape-constrained density estimation}
\end{keyword}

\end{frontmatter}\section{Introduction}
Estimation of functionals of data generating distributions is a fundamental research problem in nonparametric statistics. Consequently, substantial effort has gone into understanding estimation of nonparametric functionals (e.g., linear, quadratic, and other ``smooth" functionals) both under density and regression (specifically Gaussian white noise models) settings --- a comprehensive snapshot of this huge literature can be found in \cite{hall1987estimation}, \cite{bickel1988estimating}, \cite{donoho1990minimax1}, \cite{donoho1990minimax2}, \cite{fan1991estimation}, \cite{kerkyacharian1996estimating}, \cite{nemirovski2000}, \cite{laurent1996efficient}, \cite{cai2003note, cai2004minimax, cai2005nonquadratic} and the references therein. However, most of the above papers have focused on smoothness restrictions on the underlying (nonparametric) function class.  In contrast, parallel results for purely shape-restricted function classes (e.g., monotonicity/convexity/log-concavity) are limited. 

Estimation of a shape-restricted density is somewhat special among nonparametric density estimation problems. In shape-constrained problems classical maximum likelihood and/or least squares estimation strategies (over the class of all such shape-restricted densities) lead to consistent (see e.g., \cite{G56}, \cite{GroeneboomEtAl01}, \cite{DR09}) and adaptive estimators (see e.g.,~\cite{Birge87}, \cite{Kim16}) of the underlying density. Consequently, the resulting procedures are tuning parameter-free; compare this with density estimation techniques under smoothness constraints that usually rely on tuning parameter(s)  (e.g., bandwidth for kernel based procedures or resolution level for wavelet projections), the choice of which requires delicate care for the purposes of data adaptive implementations.

The simplest and the most well-studied shape-restricted density estimation problem is that of estimating a nonincreasing density on $[0,\infty)$. In this case, the most natural estimator of the underlying density is the Grenander estimator --- the nonparametric  maximum likelihood estimator (NPMLE) which also turns out to be the same as the least squares estimator; see e.g.,~\cite{G56},~\cite{PRao69},~\cite{G85},~\cite{GJ14}. Even for this problem only a few papers exist that study nonparametric functional estimation; see e.g.,~\cite{nickl2008uniform},~\citet[Chapter 7]{gine2015mathematical},~\cite{H14}, and \cite{sohl2015uniform}. Moreover, these papers often assume additional constraints such as: (i) smoothness, (ii) compact support, and (iii) restrictive lower bounds on the underlying true density. 

In this paper we study estimation of smooth integrated functionals of a nonincreasing density on $[0,\infty)$ using simple (tuning parameter-free) plug-in estimators. We characterize the limiting distribution of the plug-in estimator in this problem, under minimal assumptions. Moreover, we show that the limiting distribution (of the plug-in estimator) is asymptotically normal with the efficient variance for a large class of functionals, under minimal assumptions on the underlying density. {To the best of our knowledge, this is the first time that such a large class of integrated functionals based on the Grenander estimator are treated in a unified fashion}. 

\subsection{Problem formulation and a  summary of our contributions}
Suppose that $X_1,\ldots,X_n$ are independent and identically distributed (i.i.d.) random variables from a nonincreasing density $f$ on $[0,\infty)$. Let $P$ denote the distribution of each $X_i$ and $F$ the corresponding distribution function. 
The goal is to study estimation and uncertainty quantification of the functional $\tau$ defined as
\begin{align}\label{def:functional}
\tau(g,f) := \int_0^\infty g(f(x), x) dx,
\end{align}
where $g:\R^2 \to \R$ is a known (smooth) function. Functionals of the form~\eqref{def:functional} belong to the class of problems considered by many authors including~\cite{levit1978asymptotically},~\cite{bickel1988estimating},~\cite{ibragimov1991asymptotically},~\cite{birge1995estimation},~\cite{kerkyacharian1996estimating},~\cite{laurent1996efficient},~\cite{laurent1997estimation},~\cite{nemirovski2000topics} etc. and include the monomial functionals $$\int_0^\infty f^p(x)dx\quad  \mbox{ for } \quad p\geq 2.$$  Of particular significance among them is the quadratic functional corresponding to $p=2$ --- being related to other inferential problems in density models such as goodness-of-fit testing and construction of confidence balls (see e.g.,~\citet[Chapter 8.3]{gine2015mathematical} and the references therein for more details).

Let $\hat f_n$ be the Grenander estimator in this problem (see Section~\ref{section:smooth_functionals} for its characterization). Then, a natural estimator of $\tau(g,f)$ is the simple {\it plug-in estimator} 
\begin{equation}\label{eq:tau_Plugin}
\tau(g, \hat f_n) := \int_0^\infty g(\hat f_n(x), x) dx.
\end{equation} 

In this paper we characterize the asymptotic distribution of $\tau(g, \hat f_n)$ for estimating $\tau(g,f)$, when $g$ is assumed to be smooth; see Section~\ref{assumptions} for the exact assumptions on $g$. Specifically, in Theorem~\ref{thm:asymp_eff}, we show that the simple plug-in estimator $\tau(g, \hat f_n)$ is $\sqrt{n}$-consistent, and we characterize its limiting distribution. Further, when $F$ is strictly concave we show that $\tau(g, \hat f_n)$ (properly normalized) is asymptotically normal with the (semiparametric) efficient variance (see Theorem~\ref{thm:asymp_eff-2}). Moreover, we can consistently estimate this limiting variance which readily yields asymptotically valid  confidence intervals for $\tau(g,f)$. Compared to the previous results on this topic (see e.g.,~\cite{nickl2007donsker}, \cite{nickl2008uniform}, \cite{H14},~\citet[Chapter 7]{gine2015mathematical} and \cite{sohl2015uniform}) our results holds under less restrictive assumptions on the underlying density (see conditions (A1)-(A3) in Section~\ref{assumptions}) --- in particular, we do not assume that $f$ is smooth, bounded away from $0$, or compactly supported.

An important subclass of the above considered functional class is:
\begin{align}\label{def:functional-3}
\mu(h,f) := \int_0^\infty h(f(x)) dx,
\end{align}
where $h:\R \to \R$ is a known (smooth) function. In Section~\ref{section:mu_hf} we study estimation of this functional in detail. The simple plug-in estimator for estimating $\mu(h,f)$ is 
$$\mu(h,\hat f_n) := \int_0^\infty h(\hat f_n(x)) dx.$$
For the class of functionals $\mu(h,f)$, under certain assumptions on the function $h(\cdot)$, we prove in Theorem~\ref{thm:asymp_eff_mu} that the plug-in estimator is asymptotically normal with the semiparametric efficient variance, even when $F$ is not strictly concave. In fact, the asymptotic normality even holds when $f$ is piecewise constant. Finally, in the special case when $f$ is constant  (i.e., $f$ is uniform on a compact interval) we explicitly characterize the asymptotic distribution of the plug-in estimator $\mu(h,\fhat)$,  which now converges at a non-standard rate; see Theorem~\ref{theorem:uniform}. This theorem extends the results in~\cite{GP83} beyond the quadratic functional.

Finally, in Section~\ref{section:nu_hf} we focus on a special case of $\mu(h,f)$:
\begin{align}\label{def:functional-2}
\nu(h,f) := \int_0^\infty h(f(x)) f(x)dx,
\end{align}
where $h:\R \to \R$ is a known (smooth) function. As before, the simple plug-in estimator for $\nu(h,f)$ is $\nu(h,\hat f_n) := \int h(\hat f_n(x))\hat f_n(x)dx$. Another intuitive estimator in this scenario would be $\frac{1}{n} \sum_{i=1}^n h(\hat f_n(X_i))$. It is easy to show that in this case both the above plug-in estimators are exactly the same (see Lemma~\ref{lem:Basic}). Further, we show in Remark~\ref{rem:One-Step} that both these estimators are also equivalent to the standard one-step bias-corrected estimator used in nonparametric statistics. As a consequence, we show that all these three estimators are now asymptotically normal with the (semiparametric) efficient variance.

\subsection{Comparison with the literature}
 
Estimation of functionals of the form \eqref{def:functional} (or \eqref{def:functional-3} and \eqref{def:functional-2}) in the density model, for $f$ belonging to smoothness classes (e.g., Sobolev and Besov balls), has been studied in detail in the literature. Historically, the case of ``smooth" $g$ while estimating \eqref{def:functional} (or ``smooth" $h$ while estimating \eqref{def:functional-3} and \eqref{def:functional-2}) is relatively well-understood and we direct the interested reader to   \cite{hall1987estimation},~\cite{bickel1988estimating},~\cite{donoho1990minimax1},~\cite{donoho1990minimax2},~\cite{fan1991estimation},~\cite{birge1995estimation}, ~\cite{kerkyacharian1996estimating},~\cite{laurent1996efficient},~\cite{nemirovski2000},~\cite{cai2003note},~\cite{cai2004minimax},~\cite{cai2005nonquadratic},~\cite{tchetgen2008minimax} and the references therein. 

A general feature of the results for estimating ``smooth" integrated nonparametric functionals is an elbow effect in the rate of estimation based on the smoothness of the underlying function class. For example, when estimating the quadratic functional, $\sqrt{n}$-efficient estimation can be achieved as soon as the underlying density has H\"{o}lder smoothness index $\beta \ge {1}/{4}$, whereas the optimal rate of estimation is $n^{-{4\beta}/{(4\beta+1)}}$ (in root mean squared sense) for $\beta<{1}/{4}$. Since monotone densities have one weak derivative one can therefore expect to have $\sqrt{n}$-consistent efficient estimation of $\tau(g,f)$ (under suitable assumptions on $g$). However, standard methods of estimation under smoothness restrictions involve expanding the infinite dimensional function in a suitable orthonormal basis of $L_2$ and estimating an approximate functional created by truncating the basis expansion at a suitable point. The point of truncation decides the approximation error of the truncated functional as a surrogate of  the actual functional and depends on the smoothness of the functional of interest and approximation properties of the orthonormal basis used. This point of truncation is then balanced with the bias and variance of the resulting estimator and therefore directly depends on the smoothness of the function. Consequently, most optimal estimators proposed in the literature depend explicitly on the knowledge of the smoothness indices. In this regard, our main result in the paper shows the optimality of tuning parameter-free plug-in procedure based on the Grenander estimator (under less restrictive  assumptions) for estimating integral functionals of monotone densities. 

We also mention that the results of this paper are directly motivated by the general questions tackled in~\cite{bickelritov2003} and \cite{nickl2007donsker}. In particular, \cite{bickelritov2003} considered nonparametric function estimators that  satisfy the desirable ``plug-in property". An estimator is said to have the plug-in property if it is simultaneously minimax optimal in a function space norm (e.g., in $l_2$-norm) and can also be ``plugged-in'' to estimate specific classes of functionals efficiently (and/or at $\sqrt{n}$-rate). Subsequently, \cite{bickelritov2003} demonstrated general principles for constructing such nonparametric estimators pertaining to linear functionals (allowing for certain extensions) with examples arising in various problems such as nonparametric regression, survival analysis, and density estimation. Indeed, in this paper we show that the Grenander estimator satisfies such a plug-in property for estimating smooth integrated functionals of monotone densities. 

{Note that such a plug-in property of the Grenander based estimator is also illustrated in the papers~\cite{sohl2015uniform} and \cite{nickl2007donsker} where the authors study linear functionals and the entropy functional crucially using a Kiefer-Wolfowitz type uniform central limit theorem (CLT) for the Grenander estimator. Such a proof strategy however comes with the price that more restrictive assumptions need to be made on the density $f$ to derive such a CLT. In particular, the results in \cite{sohl2015uniform} and \cite{nickl2007donsker} assume a bounded domain and a uniform positive lower bound on the true density $f$ --- assumptions we forgo in this paper.} 

{\cite{H14} also considers estimation of linear functionals (and the entropy functional) based on the Grenander estimator. In~\citet[Theorem 3.1]{H14}, the limiting distribution of the Grenander based plug-in estimator is derived for estimating a linear functional (under possible misspecification, i.e., $f$ need not be nonincreasing). 
\citet[Theorem 4.1]{H14} gives the limiting distribution of the plug-in estimator when estimating the entropy functional. However, all these results assume that the support of $f$ is bounded, an assumption that we do not make in this paper. }

\subsection{Organization of the paper}
The rest of the paper is organized as follows. In Section \ref{section:prelims} we collect a few useful notation, definitions, and results which will help the presentation of the rest of the paper. In Section~\ref{assumptions}, we define the class of integrated functionals of interest along with our main assumptions.  Section~\ref{sec:Tau_g_f_n} discusses the construction of the plug-in estimator and the main results of this paper (Theorems~\ref{thm:asymp_eff} and~\ref{thm:asymp_eff-2}) which gives $\sqrt{n}$-consistency and a characterization of the asymptotic distribution of the estimator. Subsequently, in Sections \ref{section:mu_hf} and \ref{section:nu_hf} we focus on the special classes of functionals --- $\mu(h,f)$ and $\nu(h,f)$ --- and provide explicit results (Theorem~\ref{thm:asymp_eff_mu} and Corollary~\ref{cor:asymp_eff_nu}) to show that our plug-in estimator is always $\sqrt{n}$-consistent, asymptotically normal and semiparametric efficient. 
Since the case of the uniform density on a compact interval  deserves a finer asymptotic expansion, we devote Section~\ref{section:uniform} to this purpose. Section~\ref{section:numericals} presents some numerical results that validate and illustrate our theoretical findings. In Section~\ref{sec:discussion}, we discuss some potential future research problems. Finally, all the technical proofs are relegated to Section~\ref{section_appendix_proofs}.

\section{Estimation of the Integrated Functional $\tau(g,f)$}\label{section:smooth_functionals} 

\subsection{Preliminaries}\label{section:prelims} 

In this subsection we introduce some notation and definitions to be used in the rest of the paper. Throughout we let $\R_+$ denote the (compact) nonnegative real line $[0,\infty]$. The underlying probability space on which all random elements are defined is $(\Omega, \mathcal{F}, \P)$. Suppose that we have nonnegative random variables $$X_1,\ldots, X_n \stackrel{\mathrm{i.i.d.}}{\sim} P$$ having a {\it nonincreasing} density $f$ and (concave) distribution function $F$. In the rest of the paper, we use the operator notation, i.e., for any function $\psi:\R  \to \R$, $P[\psi] := \int \psi(x) dP(x)$ denotes the expectation of $\psi(X)$ under the distribution $X \sim P$. For any two functions $\gamma_1,\gamma_2: [0,\infty) \rightarrow \R$ with $\int |\gamma_1(x)|dx< +\infty$, we view $\int \gamma_2(x)d(\gamma_1(x))$ as a Lebesgue-Stieltjes integral.

Let $\P_n$ denote the empirical measure of the $X_i$'s and let $\F_n$ denote the corresponding empirical distribution function, i.e., for $x \in \R$, $$\F_n(x) := \frac{1}{n}\sum_{i=1}^n \mathbf{1}_{(-\infty, x]}(X_i).$$ For a nonempty set $T$, we let $\ell^{\infty}(T)$ denote the set of uniformly bounded, real-valued functions on $T$. Of particular importance is the space $\ell^{\infty}(\R_+)$, which we equip with the uniform metric $\|\cdot\|_\infty$ and the ball $\sigma$-field; see \citet[Chapters IV and V]{Pollard84book} for background. 

Following~\cite{BF16}, we now define the notion of the {\it least concave majorant} (LCM) operator. Given a nonempty convex set $T \subset \R_+$, the LCM over $T$ is the operator $\M_T : \ell^\infty(\R_+) \to \ell^\infty(T)$ that maps each $\xi \in \ell^\infty(\R_+)$ to the function $\M_T\xi$ where 
$$\M_T\xi(x) : = \inf \Big\{ \eta(x) : \eta \in \ell^\infty(T), \; \eta \mbox{ is concave, and } \xi \le \eta \mbox{ on } T\Big \}, \quad x \in T.$$ We write $\M$ as shorthand for $\M_{\R_+}$ and refer to $\M$ as the LCM operator. \citet[Proposition 2.1]{BF16} shows that the LCM operator $\M$ is {\it Hadamard directionally differentiable} (see e.g.,~\citet[Definition 2.2]{BF16}) at any concave function $\theta \in \ell^{\infty}(\R_+)$ tangentially to $C(\R_+)$ (here $C(\R_+)$ denotes the collection of all continuous real-valued functions on $\R_+$ vanishing at infinity). Let us denote the Hadamard directional derivative of $\M$ at the  concave function $\theta \in \ell^{\infty}(\R_+)$ by $\M_\theta'$. The following result, due to \citet[Proposition 2.1]{BF16}, characterizes the Hadamard directional  derivative of $\M_\theta'$, for $\theta \in \ell^{\infty}(\R_+)$, tangentially to $C(\R_+)$ and is crucial to our subsequent analysis.

\begin{proposition}[Proposition 2.1 of~\cite{BF16}]\label{prop:BF}
	The LCM operator $\M : \ell^\infty(\R_+) \to \ell^\infty(\R_+)$ is Hadamard directionally differentiable at any concave $\theta \in \ell^\infty(\R_+)$ tangentially to $C(\R_+)$. Its directional derivative $\M_\theta' : C(\R_+) \to \ell^\infty(\R_+)$ is uniquely determined as follows: for any $\xi \in \ell^\infty(\R_+)$ and $x \ge 0$, we have $$\M_\theta'\xi(x)=\M_{T_{\theta,x}}\xi(x),$$ where $T_{\theta,x} = \{x\} \cup U_{\theta,x}$, and $U_{\theta,x}$ is the union of all open intervals $A \subset \R_+$ such that (i) $x \in A$, and (ii) $\theta$ is affine over $A$.
\end{proposition}

Let $\hat F_n$ be the LCM of $\F_n$, i.e., $\hat F_n$ is the smallest concave function that sits above $\F_n$.  
Let $\B(\cdot)$ be the Brownian bridge process on $[0,1]$. By Donsker's theorem (see e.g.,~\citet[Theorem 19.3]{van2000asymptotic}) we know that the empirical process $ \D_n := \sqrt{n} (\F_n - F)$ converges in distribution to $\G \equiv \B \circ F$ in $\ell^\infty(\R_+)$, i.e., $$\D_n := \sqrt{n} (\F_n - F) \stackrel{d}{\to} \G.$$ 

As the delta method is valid for Hadamard directionally differentiable functions (see~\cite{Shapiro91}), as a consequence of the above discussion it follows that (see~\citet[Theorem 2.1]{BF16}), for any concave $F$, the stochastic processes $\hat \D_n := \sqrt{n} (\hat F_n - F)$ converges in distribution to $\hat \G$, 
\be
\hat \D_n := \sqrt{n} (\hat F_n - F) \stackrel{d}{\to} \hat \G, \ee 
where  
\be\label{eq:LCM_HD}
\hat \G := \M_F' \G.
\ee  
{To give more intuition about the stochastic process $\hat \G$ and the operator $\M_F'$, let us consider two simple scenarios. First, suppose that $F$ is strictly concave. Then $\hat \G = \M_F' \G = \G$ as by Proposition~\ref{prop:BF}, for every $x >0$, $$\M_F' \G(x) = \M_{\{x\}} \G(x) = \G(x);$$ also see~\citet[Proposition 2.2]{BF16}. Thus, $\D_n$ and $\hat \D_n$ converge to the same limiting object in this case. 

Next, let us suppose that $F$ is piecewise affine. In this case $f$ is piecewise constant with jumps (say) at $0<t_1<t_2<\ldots< t_k<\infty$ and values $v_1>\ldots> v_k$, for some $k \ge 2$ integer, i.e., i.e., for $x >0$,
\begin{equation}\label{eq:PW-C}
f(x) = \sum_{i=1}^k v_i \mathbf{1}_{(t_{i-1},t_i]}(x),
\end{equation} 
where $t_0 \equiv 0$. Then, 
\begin{equation}\label{eq:PW-Cons}
\M_F' \G(x) = \begin{cases} \M_{(t_{i-1},t_i)} \G(x), & \qquad \mbox{if}\;\; x \in (t_{i-1},t_i) \\ \G(x), & \qquad \mbox{if}\;\; x =t_i \;\; \mbox{for } i = 0,1,\ldots, k.\end{cases}
\end{equation}}
 
\subsection{Assumptions}\label{assumptions}

In this subsection we state the assumptions needed for our main results.
\begin{itemize}
	\item[(A1)] The true underlying density $f:[0,\infty)\rightarrow [0,\infty)$ is  nonincreasing with 
	\begin{equation*}\label{eq:f0+}
		 f(0+): = \lim_{x \downarrow 0} f(x) < +\infty. 
	\end{equation*} 
	
	\item[(A2)] We assume that for some $\alpha \in (0,1]$, 
	\begin{equation*}\label{eq:TailC}
	\quad \int_{f<1} f^{1-\alpha} (x) dx < +\infty.
	\end{equation*}
	
	\item[(A3)] Let $\s$ denote the support of $f$.  The function $g$ in \eqref{def:functional} satisfies the following conditions: 
	\begin{enumerate}
		\item [(i)] The functional \eqref{def:functional} is well-defined, i.e., $|\tau(g,f)| < +\infty$.
		
		\item [(ii)] $g(\cdot,x) \in {C^2}([0,\infty))$,  for every $x \in [0,\infty)$, where $C^2([0,\infty))$ denotes the class of all {twice} continuously differentiable functions on $[0,\infty)$. 
		
		\item [(iii)] Let $\dot{g}$ and $\ddot{g}$ denote the derivatives of $g(\cdot,\cdot)$ only with respect to the first coordinate, i.e., for $z,x\in [0,\infty)$, $$\dot{g}(z,x) = \frac{\partial}{\partial z} g(z,x) \qquad \mbox{ and } \qquad \ddot{g}(z,x) = \frac{\partial^2}{\partial z^2} g(z,x).$$ We assume that the following hold:
\begin{enumerate}
			\item [(a)] For any $B>0$ there exists a constant $K_B < \infty$ such that $$
			\sup_{x \in \s, \; z \in [0,B] }{|{\ddot{g}(z,x)}|} \le K_B.$$
			
			\item [(b)] $\dot{g}(f(\cdot),\cdot)$ is of bounded total variation on $[0,\infty)$, i.e., $$\int_{\s} |d[\dot{g}(f(x),x)]|<+\infty.$$
\end{enumerate}	
\end{enumerate}

\end{itemize}
Let us briefly discuss assumptions (A1)-(A3). Conditions (A1) and (A2) on the true density $f$ imply that $f$ is bounded and satisfies a tail decay condition. Note that (A2) is satisfied if for some $m >1$, $\limsup_{x \to \infty} x^m f(x) < \infty$; see e.g.,~\citet[Example 7.4.2]{geer2000empirical}. In particular, if $f$ is bounded and compactly supported, conditions (A1) and (A2) immediately hold. 

Condition (A3)-(i) is unavoidable since it pertains to the existence of the functional we are interested in estimating. Parts (ii) and (iii,a) of (A3) pertain to the function $g(\cdot,\cdot)$ and quantifies the notion of the smooth functional we consider. Finally, part (iii,b) of (A3) is slightly stronger than the requirement that $\int \dot{g}(f(x),x)f(x)dx$ exists --- which in turn appears in the information bound for estimating $\tau(g,f)$ and is therefore needed to be finite for efficient estimators to exist. More precisely, the finite total variation of $\dot{g}(f(x),x)$ helps us define several Lebesgue-Stieltjes integrals with respect to $\dot{g}(f(x),x)$ without further conditions.  

Observe that our assumption (A3) implies twice (right) differentiability of $g(\cdot,x)$ (for every $x$) at $0$ and hence does not include the entropy functional corresponding to $g(z,x)=-z\log{z}$, unless $f$ is assumed to be bounded away from $0$.

\subsection{$\sqrt{n}$-Consistency of $\tau(g,\hat f_n)$}\label{sec:Tau_g_f_n}
As mentioned before, a natural estimator of $f$ in this situation is $\hat f_n$, the Grenander estimator --- the (nonparametric) maximum likelihood estimator of a nonincreasing density on $[0,\infty)$; see~\cite{G56}. The Grenander estimator $\hat f_n$ is defined as the left-hand slope of $\hat F_n$, the LCM of the empirical distribution function $\F_n$. Note that $\hat F_n$ is piecewise affine and a valid {\it concave} distribution function by itself. Further, $\hat f_n$ is a piecewise constant (nonincreasing) density with possible jumps only at the data points $X_i$'s. Moreover, $\hat f_n(x) = 0$ for $x > \max\limits_{i=1,\ldots, n} X_i$. Consequently, a natural plug-in estimator of $\tau(g, f)$ is 
$$\tau(g, \hat f_n) := \int_0^\infty g(\hat f_n(x), x) dx.
$$ 
We study estimation and uncertainty quantification of the functional $\tau(g, \hat f_n)$. The following result shows that the plug-in estimator $\tau(g, \hat f_n)$ is $\sqrt{n}$-consistent and explicitly characterizes its asymptotic distribution. 
\begin{theorem}\label{thm:asymp_eff} 
Assume that conditions (A1)-(A3) hold. Then, 
\begin{align}\label{eq:Eff}
	\sqrt{n}\left(\tau(g,\hat{f}_n)-\tau(g,f)\right) \stackrel{d}{\to}  - \int_0^\infty \hat{\mathbb{G}}(x) \, d[\dot{g}(f(x), x)] =: Y
\end{align}
where $\hat{\mathbb{G}}(\cdot)$ is defined in~\eqref{eq:LCM_HD}.  
\end{theorem}
We defer the proof of Theorem~\ref{thm:asymp_eff} to Section~\ref{sec:Steps} and before proceeding further, make a few comments regarding the implications of Theorem \ref{thm:asymp_eff}.

\begin{remark}[Random variable $Y$] The limiting random variable $Y$ in the statement of Theorem \ref{thm:asymp_eff} is well-defined. Note that $Y = - \int_\s \hat{\mathbb{G}}(x) \, d[\dot{g}(f(x), x)]$ as $\hat{\mathbb{G}}$ is zero outside $\s$, the support of the distribution $P$. This follows from the following facts: (i) $\hat{\mathbb{G}}$ is a (measurable) random element in $\ell^{\infty}(\R_+)$ (see e.g.,~\cite{BF16}); (ii) by assumption (A3)-(iii,b), for any $\phi\in \ell^{\infty}(\R_+)$, the map $\phi \mapsto \int_\s \phi(x)d[\dot{g}(f(x),x)]$ is continuous and measurable (from $\ell^{\infty}(\R_+)$ to $\mathbb{R}$).
\end{remark}

\begin{remark}[Assumptions in Theorem~\ref{thm:asymp_eff}]
Note that Theorem~\ref{thm:asymp_eff} is valid under much weaker assumptions on $f$ than previously studied in the literature. Although, similar results have been derived in  \cite{nickl2008uniform} and~\citet[Chapter 7]{gine2015mathematical}, the derivations relied on further smoothness, lower bound, and compact support assumptions on $f$. In contrast, our only technical assumptions (A1) and (A2) are significantly less restrictive and only demand $f$ to be bounded and have a polynomial-type tail decay.
\end{remark}

\begin{remark}[Connection to estimation in smoothness classes] For smoothness classes of functions it is well-known that $\sqrt{n}$-consistent estimation of $\tau(g,f)$ is possible if $f$ has more than ${1}/{4}$ derivatives (see e.g., \cite{birge1995estimation}) and a typical construction of such an estimator proceeds via a bias-corrected one-step estimator; see e.g.,~\cite{bickel1988estimating},~\cite{robins2008higher},~\cite{robins2015higher}. Since monotone densities have one weak derivative, one can expect $\sqrt{n}$-consistent estimation of $\tau(g,f)$. Theorem~\ref{thm:asymp_eff} shows that in fact the plug-in principle based on the Grenander estimator is $\sqrt{n}$-consistent with a distributional limit without any further assumptions on $f$. It will be further shown in Section \ref{section:nu_hf} that for estimating the special functional $\nu(h,f)$, defined in~\eqref{def:functional-2}, the one-step bias-corrected estimator is equivalent to the plug-in estimator. Consequently, this will provide more intuition on $\sqrt{n}$-consistent efficient estimation of $\nu(h,f)$, without further assumptions.  
\end{remark}

\begin{remark}[Tuning parameter-free estimation]
Being based on the Grenander estimator and the plug-in principle, the estimator $\tau(g,\hat{f}_n)$ is completely tuning parameter-free unlike those considered in~\cite{robins2008higher},~\cite{mukherjee2016adaptive},~\cite{mukherjee2017semiparametric}. 
\end{remark}
	
In the following result (proved in Section~\ref{sec:Thm-Normal-Limit}), we show that when $F$ is strictly concave, the plug-in estimator $\tau(g, \hat f_n)$ is $\sqrt{n}$-consistent and asymptotically normal with the (semiparametric) efficient variance. 
\begin{theorem}\label{thm:asymp_eff-2} 
Suppose that $F$ is strictly concave on $[0,\infty)$. Then, under the assumptions of Theorem~\ref{thm:asymp_eff}, we have 
\begin{align}\label{eq:Eff-2}
	Y = - \int_0^\infty {\mathbb{G}}(x) \, d[\dot{g}(f(x), x)] \stackrel{d}{=} N(0,\sigma^2_{\mathrm{eff}}(g,f)),
	\end{align}
where, for $X \sim P$, 
\begin{equation}\label{eq:sigma}\sigma^2_{\mathrm{eff}}(g,f) := \mathrm{Var}\Big(\dot{g}(f(X), X)\Big) < +\infty.
\end{equation}
\end{theorem}

In general, when $F$ is not strictly concave, we cannot say if $Y$ (described by \eqref{eq:Eff}) admits a simpler description. Indeed, the integral $\int \hat{\mathbb{G}}(x)d[\dot{g}(f(x), x)]$ can be a highly non-linear functional of the Gaussian process ${\mathbb{G}}$ and consequently there is no immediate reason to believe that the limiting distribution is normal. Surprisingly though, for a special class of functionals, namely $\mu(h,f)$ introduced in~\eqref{def:functional-3}, the distribution of $\int  \hat{\mathbb{G}}(x)d[\dot{g}(f(x), x)]$ is always Gaussian. The next section is therefore devoted to the complete understanding of this special case.

\begin{remark}[Linear functional] {Theorem \ref{thm:asymp_eff}  covers the case of estimating an integrated linear functional $\int w(x) f(x) dx$ where $w(\cdot)$ is any weight function with finite variation on $\s$. Indeed, this follows by taking $g(z,x) = w(x) z$ in Theorem~\ref{thm:asymp_eff}. 

When $F$ is strictly concave, Theorem~\ref{thm:asymp_eff-2} shows that the limiting distribution of the plug-in estimator is asymptotically normal. When $F$ has a density $f$ which is piecewise constant and has the form~\eqref{eq:PW-C} and $w(\cdot)$ is continuous, Theorem~\ref{thm:asymp_eff} along with~\eqref{eq:PW-Cons} yields $$\sqrt{n} \left\{\int w(x) \hat f_n(x) dx - \int w(x) f(x) dx \right\}  \stackrel{d}{\to} - \sum_{i=1}^k \int_{t_{i-1}}^{t_i} \M_{(t_{i-1},t_i)} \G(x) dw(x);$$
cf.~\cite{H14} where the author derives the limiting distribution for linear functionals (under model-misspecification) and expresses the limit in a different form.}
\end{remark}

\section{Plug-in Efficient Estimation of $\mu(h,f)$}\label{section:mu_hf}
In this section we focus on the special case of estimating $\mu(h,f)$ (as defined in~\eqref{def:functional-3}) where $h:[0,\infty) \to \R$ is assumed to be a known function satisfying:
\begin{itemize}
	\item[(A4)] $\hspace{2in} h \in {C^2}([0, \infty))$, \newline where ${C^2}([0, \infty))$ denotes the class of all {twice} continuously differentiable functions on $[0,\infty)$ and $$\int |d [h'(f(x))]| < + \infty.$$
\end{itemize}

The natural plug-in estimator of $\mu(h,f)$ is 
$$ \mu(h, \hat f_n) := \int_0^\infty h(\hat f_n(x)) dx.$$ 
Theorem~\ref{thm:asymp_eff} can be used to obtain the $\sqrt{n}$-consistency and the asymptotic distribution of $\mu(h,\hat f_n)$ by considering $g(z,x) = h(z)$. Note that for $h(\cdot)$ satisfying (A4), the corresponding $g$ automatically satisfies (A3) defined in Section~\ref{assumptions}. However, in this case we can simplify the form of the limiting distribution $Y$ (see~\eqref{eq:Eff}) and prove that it is indeed a mean zero normal random variable with the efficient variance. Theorem~\ref{thm:asymp_eff_mu} (proved in Section~\ref{sec:Reduction}) formalizes this.
 \begin{theorem}\label{thm:asymp_eff_mu} 
	Assume that conditions (A1)-(A2) and (A4) hold. Then, 
	\begin{align}\label{eq:Eff_mu}
	\sqrt{n}\left(\mu(h,\hat{f}_n)-\mu(h,f)\right) \stackrel{d}{\to} 
	N(0,\sigma^2_{\mathrm{eff}}(h,f)),
	\end{align}
	where, for $X \sim P$, $\sigma^2_{\mathrm{eff}}(h,f) := \mathrm{Var}\Big(h'(f(X))\Big) < +\infty$.
\end{theorem}
A few remarks about the implications of  Theorem~\ref{thm:asymp_eff_mu} are in order now. 

\begin{remark}[Assumptions in Theorem~\ref{thm:asymp_eff_mu}]
Assumption (A4) in Theorem~\ref{thm:asymp_eff_mu} is implied by assumption (A3) where one considers $g(z,x)=h(z)$ with $h\in C^2([0,\infty))$. Thus, the $\sqrt{n}$-consistency and existence of asymptotic distribution for the plug-in principle is automatic from Theorem \ref{thm:asymp_eff}. However, unlike Theorem \ref{thm:asymp_eff-2}, the requirement of $F$ being strictly concave is no longer necessary for the validity of Theorem \ref{thm:asymp_eff_mu}. 
\end{remark}

\begin{remark}[Asymptotic efficiency] 
Observe that the limiting variance in Theorem~\ref{thm:asymp_eff_mu}  ({see~\eqref{eq:Eff_mu}}) matches the nonparametric efficiency bound in this problem (\cite{bickel1993efficient}, ~\cite{van2000asymptotic},~\cite{bolthausen2002lectures},~\cite{laurent1996efficient}). Indeed, this efficiency bound corresponds to the case when the maximal  tangent space in the model is the whole of $L_2(F)$. Although such a bound can be shown to hold for models where $f$ has a lower bound on its slope (in absolute value) and is compactly supported, it is not immediate for more general monotone densities. Consequently, the statement that the ``plug-in estimator $\mu(h,\hat{f}_n)$ is asymptotically efficient for any concave $F$ (under assumptions (A1), (A2), and (A4))" should be understood in the sense of simply achieving the efficiency bound in the nonparametric model. When the tangent space is  $L_2(F)$, this is indeed the semiparametric efficiency bound.
\end{remark}

\begin{remark}[Monomial functionals] 
Of special interest is the case when $h(x) = x^{p}$, where $p\geq 2$ is an integer. For example, the quadratic functional $\int f^2$ is obtained when $p=2$. In this case, Theorem~\ref{thm:asymp_eff_mu} yields an asymptotically efficient variance of
		\begin{equation}\label{eq:sigma_p}\sigma^2_{\mathrm{eff}}(h,f) =\mathrm{Var}_P\left(pf(X)^{p-1}\right)=p^2\left(\int_0^\infty f^{2p-1}(x)dx-\left(\int_0^\infty f^{p}(x)dx \right)^2\right),
\end{equation}
for any concave $F$. 
\end{remark}

\begin{remark}[Construction of confidence intervals]  
As an important consequence of Theorem~\ref{thm:asymp_eff_mu} one can construct asymptotically valid confidence intervals for the functional $\mu(h,f)$. In particular, one can estimate the efficient asymptotic variance $\sigma^2_{\mathrm{eff}}$ as follows. Note that, $$\sigma^2_{\mathrm{eff}}(h,f)= \mathrm{Var}\Big(h'(f(X))\Big)=\mu(h_1,f)-(\mu(h_2,f))^2 $$ where $h_1(z):=z (h'(z))^2$ and $h_2(z):=z h'(z)$ for $z \ge 0$. As a result, if $h\in C^3([0,\infty))$ then $h_1,h_2\in C^2([0,\infty))$ and Theorem~\ref{thm:asymp_eff_mu} implies that $\hat{\sigma}^2_n:=\mu(h_1,\hat{f}_n)-(\mu(h_2, \hat{f}_n))^2 $ is a consistent estimator of $\sigma^2_{\mathrm{eff}}(h,f)$. Consequently $\mu(h,\fhat)\pm z_{\alpha/2}\hat{\sigma}_n$ (with $z_{\alpha/2}$ being the $1-\alpha/2$ quantile of the standard normal distribution) is an asymptotically valid $100\times (1-\alpha)\%$ confidence interval for $\mu(h,f)$. Indeed, the additional assumption of $h\in C^3([0,\infty))$ compared to $h\in C^2([0,\infty))$ in Theorem~\ref{thm:asymp_eff_mu} can be reduced since we only demand consistency of $\hat \sigma_n$, instead of asymptotic normality.  We omit the details for the sake of avoiding repetitions. Finally, the validity of this confidence interval is pointwise (true for every fixed underlying $f$) and not in a uniform sense. 
\end{remark}

\begin{remark}[Uniform distribution]\label{rem:Unif} If $f$ is the uniform distribution on $[0,c]$, for any $c>0$, Theorem~\ref{thm:asymp_eff-2} implies that $\sqrt{n}(\mu(h,\hat{f}_n)-\mu(h,f)) \stackrel{\P}{\to} 0$, as $\sigma^2_{\mathrm{eff}}(h,f) = 0$ in~\eqref{eq:Eff_mu}. In Section~\ref{section:uniform} below we deal with this scenario (i.e., when $f$ is uniform) and obtain a non-degenerate distributional limit, after proper normalization.
\end{remark}

\subsection{When $f$ is uniformly distributed}\label{section:uniform}
As mentioned in Remark~\ref{rem:Unif}, Theorem~\ref{thm:asymp_eff_mu} does not give a non-degenerate limit when $f$ is the uniform distribution on the interval $[0,c]$, for any $c >0$. Indeed, as we will see, in such a case $\mu(h,\hat f_n)$ has a faster rate of convergence. In this subsection we focus our attention to the case when $f$ is the uniform distribution on $[0,1]$; an obvious scaling argument can then be used to generalize the result to the case when $f$ is uniform on $[0,c]$, for any $c >0$.

Suppose that $P$ is the uniform distribution on $[0,1]$, i.e., $f(x) = \mathbf{1}_{[0,1]}(x)$, $x \in \R$. It has been shown in~\cite{GP83} (also see~\cite{G83}) that in this case 
\be\label{eq:Unif-2}
\frac{n \int_0^1 \big(\hat f_n^2(x) - 1\big) dx - \log n}{\sqrt{3 \log n}} \stackrel{d}{\to} N(0,1).
\ee
The above result yields the asymptotic distribution for the plug-in estimator of the quadratic functional (as $\int f^2(x) dx = 1$), properly normalized. In the following theorem we extend the above result to general smooth functionals of the Grenander estimator. 
\begin{theorem}\label{theorem:uniform}
Let $X_1,\ldots,X_n$ be {i.i.d.}~Uniform$([0,1])$. Suppose that $h\in C^4([0,\infty))$, where $C^4([0,\infty))$ is the space of all four times continuously differentiable functions on $[0,\infty)$. Then
\be \label{eq:Unif-Lim}
\frac{n\left(\mu(h,\hat{f}_n)-\mu(h,f)\right)-\frac{1}{2}h''(1)\log{n}}{\sqrt{3[\frac{1}{2}h''(1)]^2\log{n}}} \stackrel{d}{\rightarrow} N(0,1). 
\ee
\end{theorem}
The proof of the above result is given in Section~\ref{pf:Thm-Unif} and follows closely the line of argument presented in~\cite{GP83}. Indeed, as is somewhat apparent from the statement of the theorem, the proof relies on a Taylor expansion of  $h$ and thereafter controlling the error terms. In particular, we prove that  
$$ 
\frac{n\int_0^1 (\fhat(x)-1)^k dx}{\sqrt{\log{n}}}=o_{\P}(1), \quad k=3,4.$$ and then a simple three term Taylor expansion yields the desired result.

\section{Other Efficient Estimators of $\nu(h,f)$}\label{section:nu_hf}
In this section we focus on the special case of estimating the functional $\nu(h,f)$ (as defined in~\eqref{def:functional-2}) where $h:[0,\infty) \to \R$ is assumed to be a known function and consider other (efficient) estimators of $\nu(h,f)$. Our plug-in estimator of $\nu(h,f)$ is 
\begin{equation}\label{Eq:Est}
\nu(h, \hat f_n) := \int h(\hat f_n(x)) \hat f_n(x) dx = 
\hat P_n[h\circ \hat f_n],
\end{equation} 
where $\hat P_n$ denotes the probability measure associated with the Grenander estimator (i.e., $\hat P_n$ has distribution function $\hat F_n$). 

We first contrast the above plug-in estimator with two other natural estimators and argue that the special structure of the Grenander estimator implies their equivalence. Observe that another natural estimator of $\nu(h,f) \equiv P[h \circ f]$ would be $$\P_n[h \circ \hat f_n] := \frac{1}{n} \sum_{i=1}^n h(\hat f_n(X_i)).$$
Indeed, Lemma \ref{lem:Basic} below shows that this natural estimator is exactly the same as the plug-in estimator $\hat P_n[h\circ \hat f_n]$. 
\begin{lemma}\label{lem:Basic} For any function $h:\R \to \R$,
$\hat P_n[h\circ \hat f_n] = \P_n[h \circ \hat f_n]$.
\end{lemma}

\begin{remark}[One-step estimator]\label{rem:One-Step}
There is however another interesting by-product of Lemma \ref{lem:Basic}: The plug-in estimator $\hat f_n$ is also equivalent to the classically studied one-step estimator which is traditionally efficient for estimating integrated functionals, such as $\nu(h,f)$, over smoothness classes for $f$; see e.g.,~\cite{bickel1993efficient},~\cite{van2000asymptotic},~\cite{bolthausen2002lectures}. In particular, a first order influence function of $\nu(h,f)$ at $P$ is $P[h \circ f+f (h' \circ f)]$ and consequently a one-step estimator obtained as a bias-corrected version of a $\tilde{f}$-based plug-in estimator (here $\tilde{f}$ is any estimator of $f$) is given by 
\begin{align*}
\tilde{\nu} :=\nu(h,\tilde{f})+\P_n\left[h \circ \tilde f +\tilde{f}( h' \circ \tilde f)-\tilde{P}[h \circ \tilde f +\tilde f (h'\circ \tilde  f) ]\right]
\end{align*}
where $\tilde{P}$ denotes the probability measure associated with $\tilde{f}$.
However, Lemma \ref{lem:Basic} implies that when $\tilde{f}=\fhat$ (the Grenander estimator)
$$\P_n\left[h \circ \tilde f +\tilde{f}(h' \circ \tilde f)-\tilde{P}[h \circ \tilde f +\tilde f(h' \circ \tilde f)]\right]=0 \qquad \Rightarrow \quad \tilde{\nu}=\nu(h,\fhat).$$ 
Consequently, in this case of estimating a monotone density, the first order influence function based one-step estimator, obtained as a bias-corrected version of a Grenander-based plug-in estimator, coincides with the simple plug-in estimator~\eqref{Eq:Est}. 
\end{remark}

Since all the three intuitive estimators in this problem turn out to be equivalent, one can expect them to be efficient as well (in a semiparametric sense). Indeed this is the case. The following result, an immediate consequence of Theorem~\ref{thm:asymp_eff_mu}, shows this.

\begin{corollary}\label{cor:asymp_eff_nu} 
	Assume that conditions (A1)-(A2) and (A4) hold. 
	Then, 
\begin{equation}\label{eq:Eff_nu}
	\sqrt{n}\left(\nu(h,\hat{f}_n)-\nu(h,f)\right) \stackrel{d}{\to}
	N(0,\sigma^2_{\mathrm{eff}}(h,f)),
\end{equation}
	where, for $X \sim P$, $\sigma^2_{\mathrm{eff}}(h,f) := \mathrm{Var}\Big(h'(f(X))f(X)+h(f(X))\Big) <\infty$.
\end{corollary}

\section{Simulation study}\label{section:numericals} 
In this section we illustrate the distributional convergence of our plug-in estimator $\tau(g,\hat f_n)$ (see~\eqref{eq:tau_Plugin}) for estimating $\tau(g,f)$. Let us consider the case of estimating the quadratic functional, i.e., $\tau(g,f) = \int f^2(x) dx $ which corresponds to $g(z,x) = z^2$. 

\begin{figure}[!ht]
	\centering
	\includegraphics[width=.325\textwidth, height= .2\textheight]{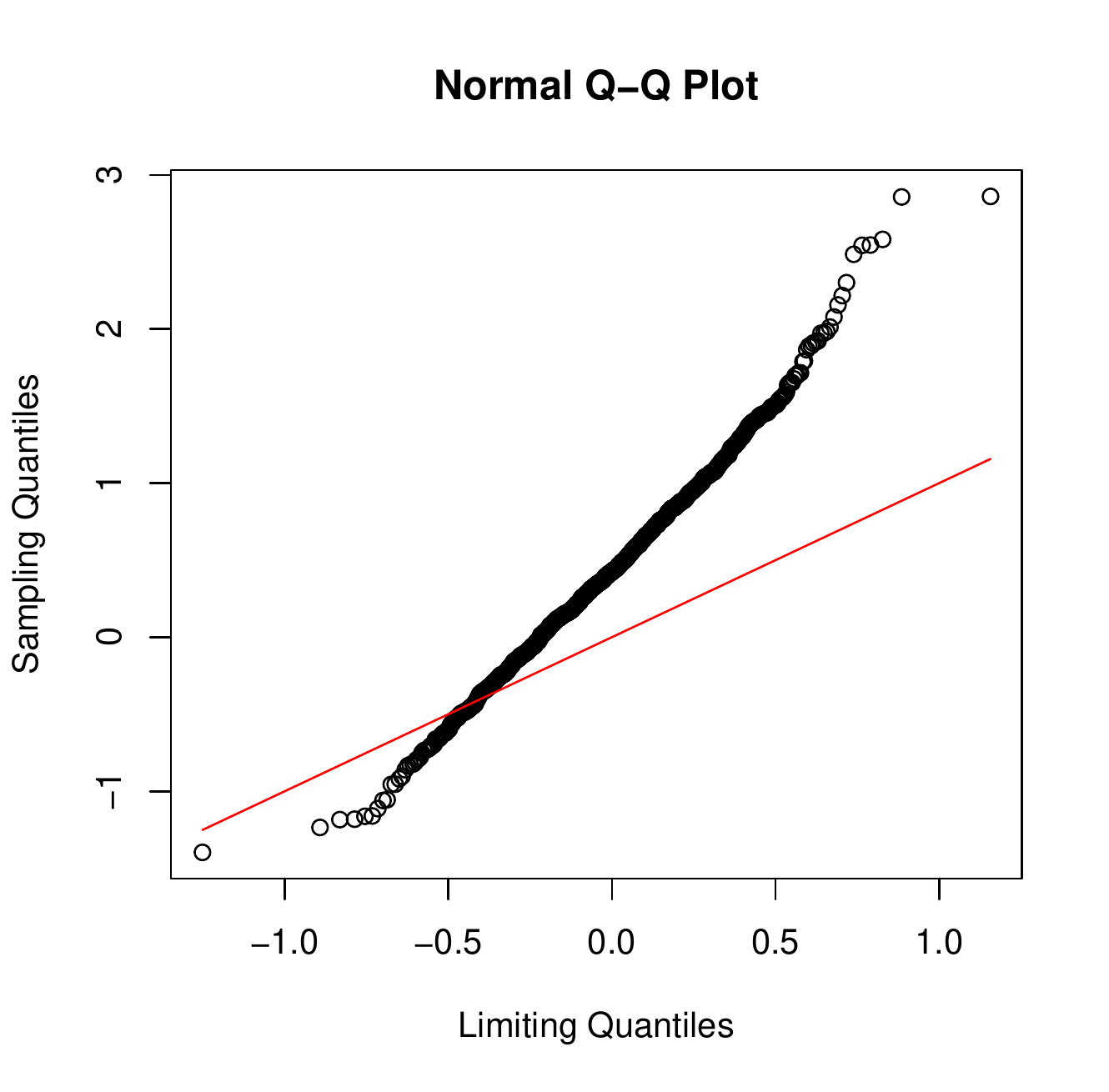}
	\includegraphics[width=.325\textwidth, height= .2\textheight]{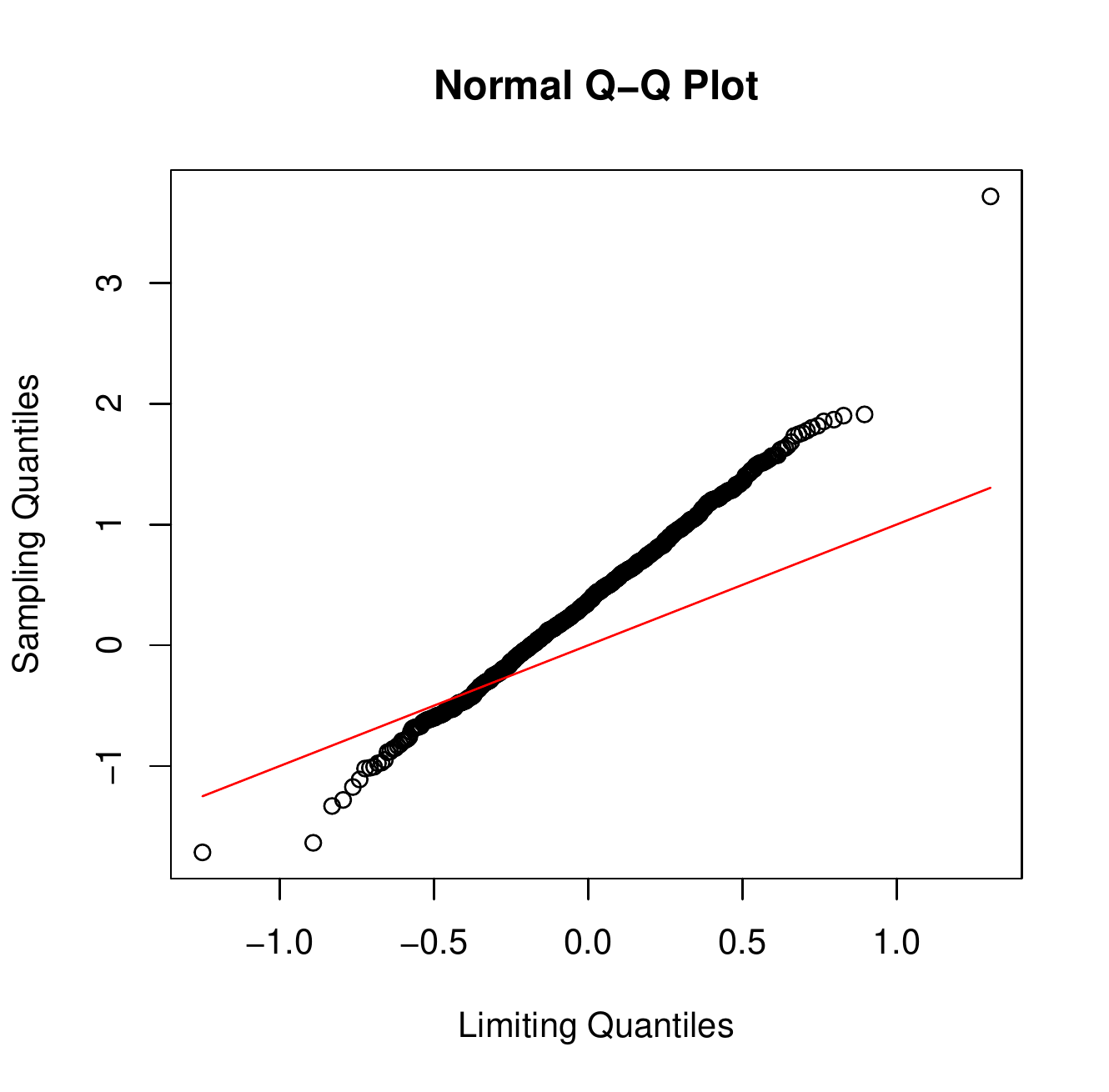}
	\includegraphics[width=.325\textwidth, height= .2\textheight]{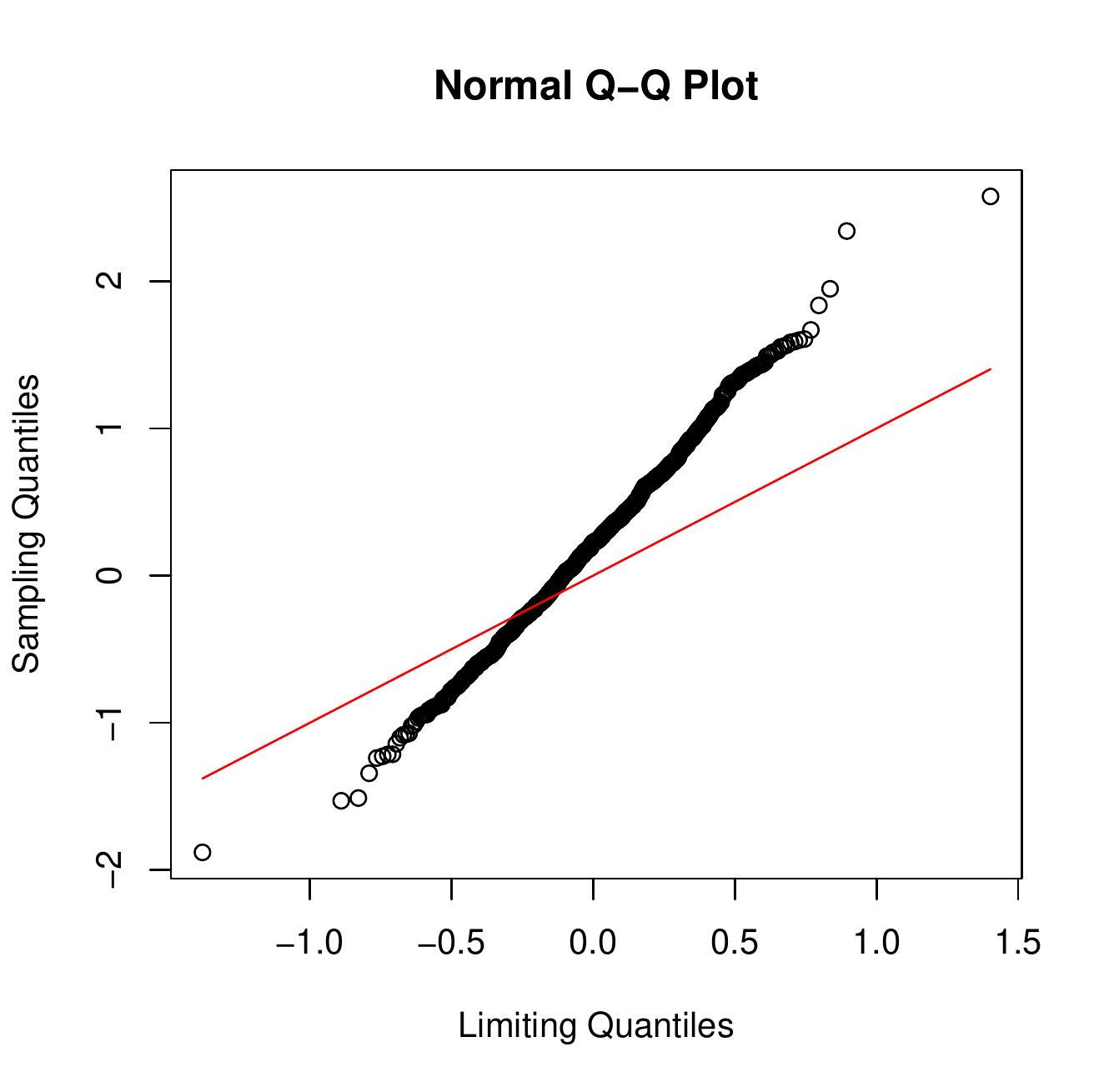}
	\caption{The Q-Q plots compare the sampling distribution of $\sqrt{n}\{\tau(g,\hat f_n) - \tau(g,f)\}$, for $g(z,x) = z^2$, against the limiting normal distribution $N(0, \sigma^2_{\mathrm{eff}}(f))$ (see Theorem~\ref{thm:asymp_eff-2}) when $F = $ Exponential(1). The three plots correspond to sample sizes $n=5000, 20000$, and 100000. Drawn in red is the $y=x$ line, which illustrates  the difference between the finite sample distribution and the limiting distribution.}
	\label{fig:Exp}
\end{figure}

We first consider the case when the true density $f$ is strictly concave and thus, the asymptotic distribution of $\tau(g,\hat f_n)$ is given by~\eqref{eq:Eff-2} (in Theorem~\ref{thm:asymp_eff-2}) where $\sigma^2_{\mathrm{eff}}(f)$ is given in~\eqref{eq:sigma_p} (with $p=2$). For the Q-Q plots in Figure~\ref{fig:Exp} we took $f$ to be the exponential density with parameter 1 for which $\tau(g,f) = 1/2$ and $\sigma^2_{\mathrm{eff}}(f) = 4/12$. The plots show the sampling distribution of $\sqrt{n}\{\tau(g,\hat f_n) - \tau(g,f)\}$ as the sample size increases ($n= 5000, 20000, 100000$). The sampling distributions is approximated from 1000 independent replications. The sampling distribution of $\sqrt{n}\{\tau(g,\hat f_n) - \tau(g,f)\}$, as $n$ increases, seems to converge to $N(0, \sigma^2_{\mathrm{eff}}(f))$, but even for moderate sample sizes there is a non-negligible bias. Although the Q-Q plots in Figure~\ref{fig:Exp} show a visible deviation between the sampling quantiles and the limiting quantiles, the quantiles lie approximately on a straight line suggesting that the sampling distribution is well-approximated by a normal (with a non-zero mean). Although these findings broadly corroborate the theoretical result in Theorem~\ref{thm:asymp_eff-2}, it is indeed true that the asymptotic regime seems to kick in quite slowly.
\begin{figure}[!ht]
	\centering
	\includegraphics[width=.325\textwidth, height= .2\textheight]{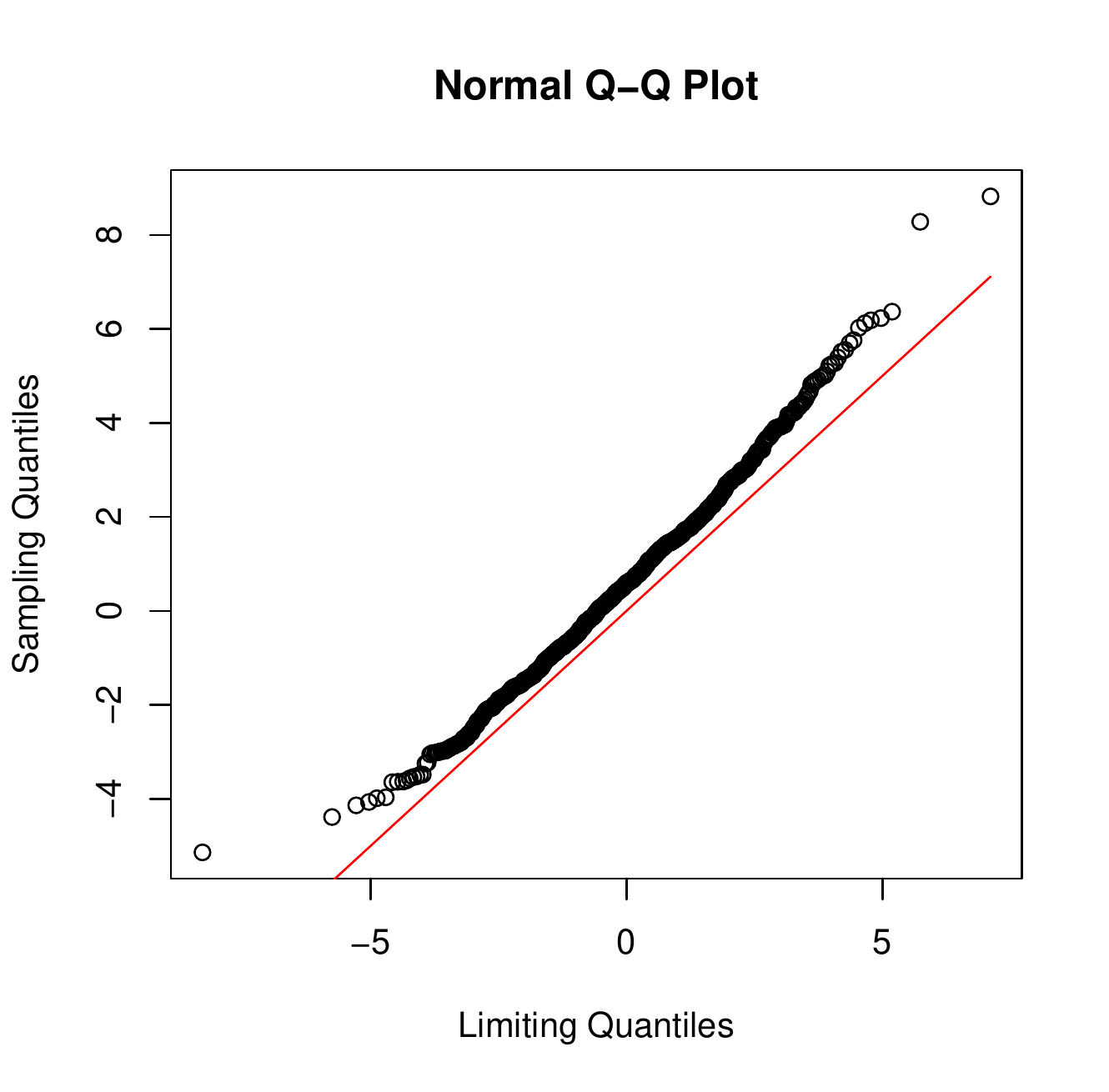}
	\includegraphics[width=.325\textwidth, height= .2\textheight]{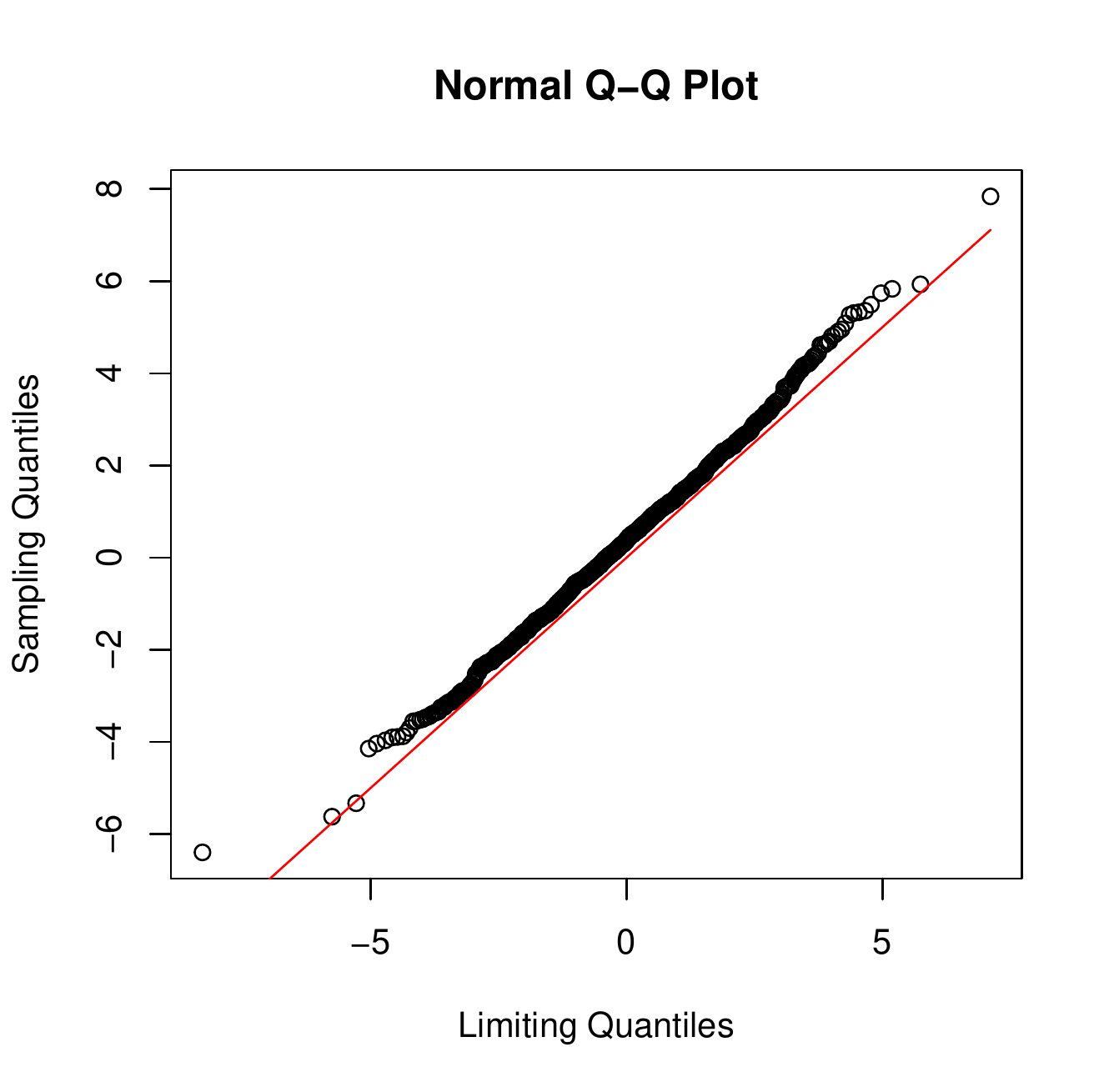}
	\includegraphics[width=.325\textwidth, height= .2\textheight]{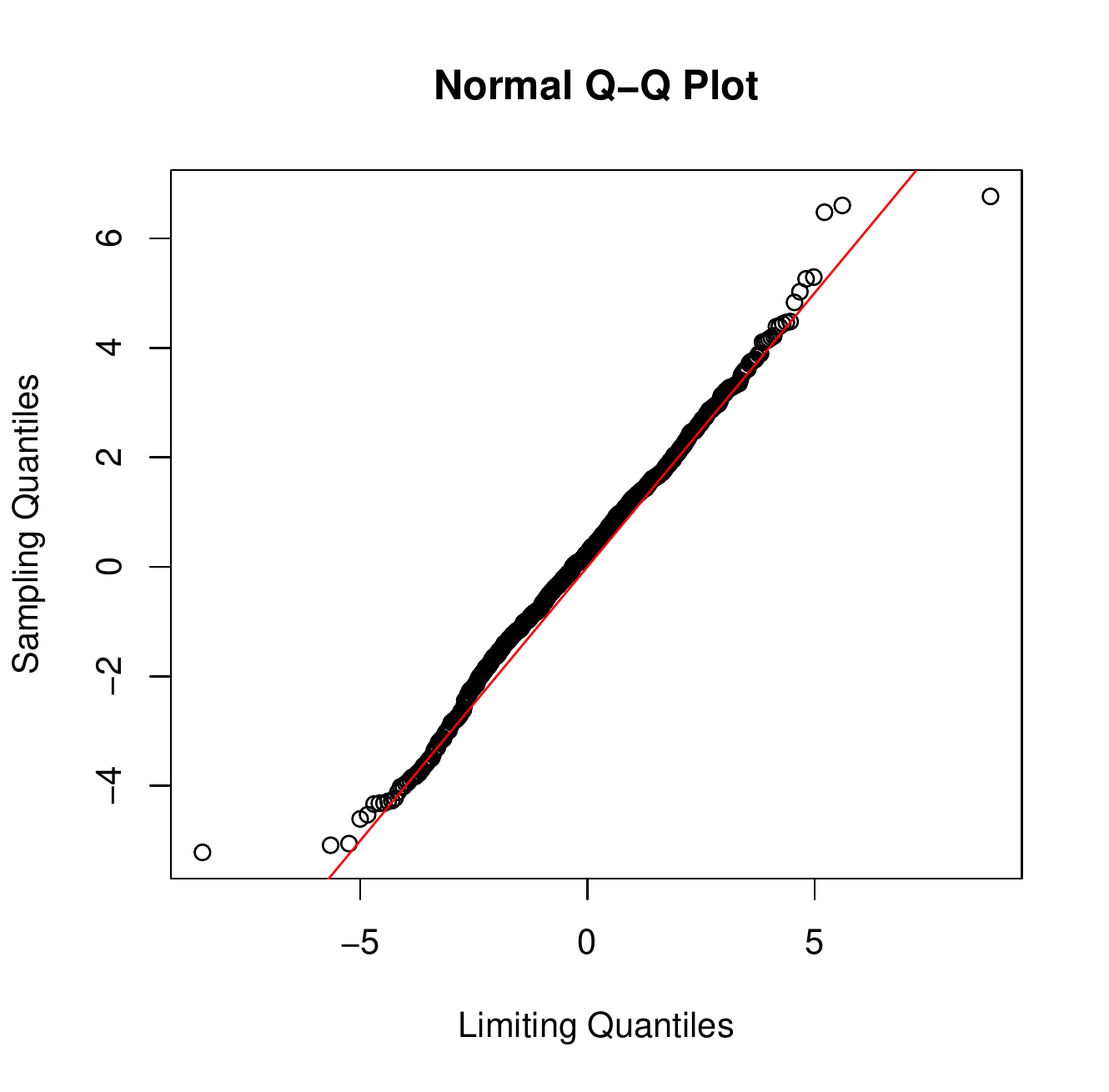}
	\caption{The same plots as in Figure~\ref{fig:Exp} for the case when $F$ is given in~\eqref{eq:F-PWA}.}
	\label{fig:PWA-2}
\end{figure}

Let us now consider the case when $f$ is piecewise constant. To fix ideas, as in~\cite{BF16}, let us take (with $c := 1 - 1/\sqrt{2}$) 
\begin{equation}\label{eq:F-PWA}
F(x) :=  \begin{cases} \frac{x}{\sqrt{2} -1}, & x \in [0,c] \\ 2 - \sqrt{2} + (\sqrt{2} - 1) x, & x \in [c,1]. \end{cases} 
\end{equation}
Then, from~\citet[Proposition 2.1]{BF16}, it follows that, for any $\xi \in C(\R_+)$, 
$$\M_F' \xi(x) = \begin{cases} \M_{[0,c]} \xi(x), & x \in [0,c] \\ 
\M_{[c,1]} \xi(x), & x \in [c,1]. 
\end{cases} $$
Theorem~\ref{thm:asymp_eff_mu} gives the asymptotic distribution of $\tau(g,\hat f_n) = \mu(h,\hat f_n) = \int f^2(x) dx$ (where $h(x) = x^2$) in this case. Figure~\ref{fig:PWA-2} shows the corresponding Q-Q plots (for sample sizes $n= 5000, 20000, 100000$) when $F$ is defined in~\eqref{eq:F-PWA} for which $\tau(g,f) \approx 1.828$ and $\sigma^2_{\mathrm{eff}}(f) \approx 3.314$. 
We can see that the sampling distribution of $\sqrt{n}\{\tau(g,\hat f_n) - \tau(g,f)\}$ converges to the desired limiting normal distribution. In this case, the normal approximation seems quite good even for moderately small sample sizes.

Let us now provide some numerical evidence to illustrate that the limiting distribution in~\eqref{eq:Eff}  (in Theorem~\ref{thm:asymp_eff}) need not always be normal when estimating a general functional of the form $\tau(g,f)$ (see~\eqref{def:functional}). In this simulation study we consider the case when $f$ is piecewise constant; in particular, we take $F$ as defined in~\eqref{eq:F-PWA}. Consider estimating the functional $\tau(g,f) = \int x f^2(x) dx$ (i.e., $g(z,x) = x z^2$). The Q-Q plots in Figure~\ref{fig:PWA} show the sampling distribution of 
$\sqrt{n}\{\tau(g,\hat f_n) - \tau(g,f)\}$ against the quantiles of a mean zero normal distribution with the efficient variance   (which can be computed using~\eqref{eq:sigma}) as the sample size varies ($n=5000, 20000$ and $100000$). The Q-Q plots indicate several interesting features: (i) the sampling quantiles show a non-linear trend and is quite different from the theoretical quantiles (assuming the normal limit); (ii) the sampling distribution does not seem to change with $n$, suggesting that the sampling distribution is already close to its asymptotic limit (which is given by~\eqref{eq:Eff}). Thus, we conjecture that in this scenario, the sampling distribution of $\sqrt{n}\{\tau(g,\hat f_n) - \tau(g,f)\}$ does not converge to a normal limit with (mean 0 and) the efficient variance.

\begin{figure}[!ht]
	\centering
	\includegraphics[width=.325\textwidth, height= .2\textheight]{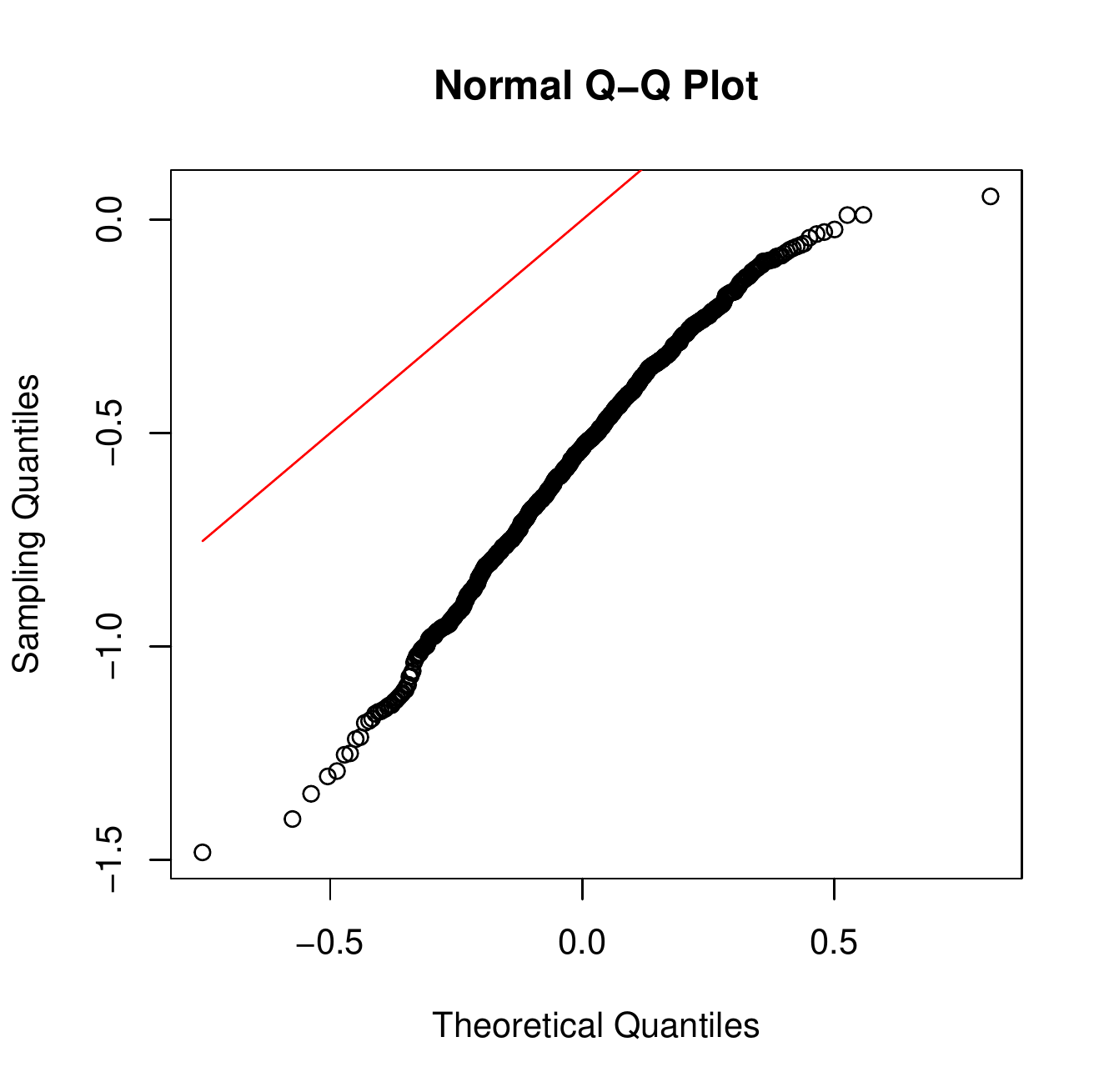}
	\includegraphics[width=.325\textwidth, height= .2\textheight]{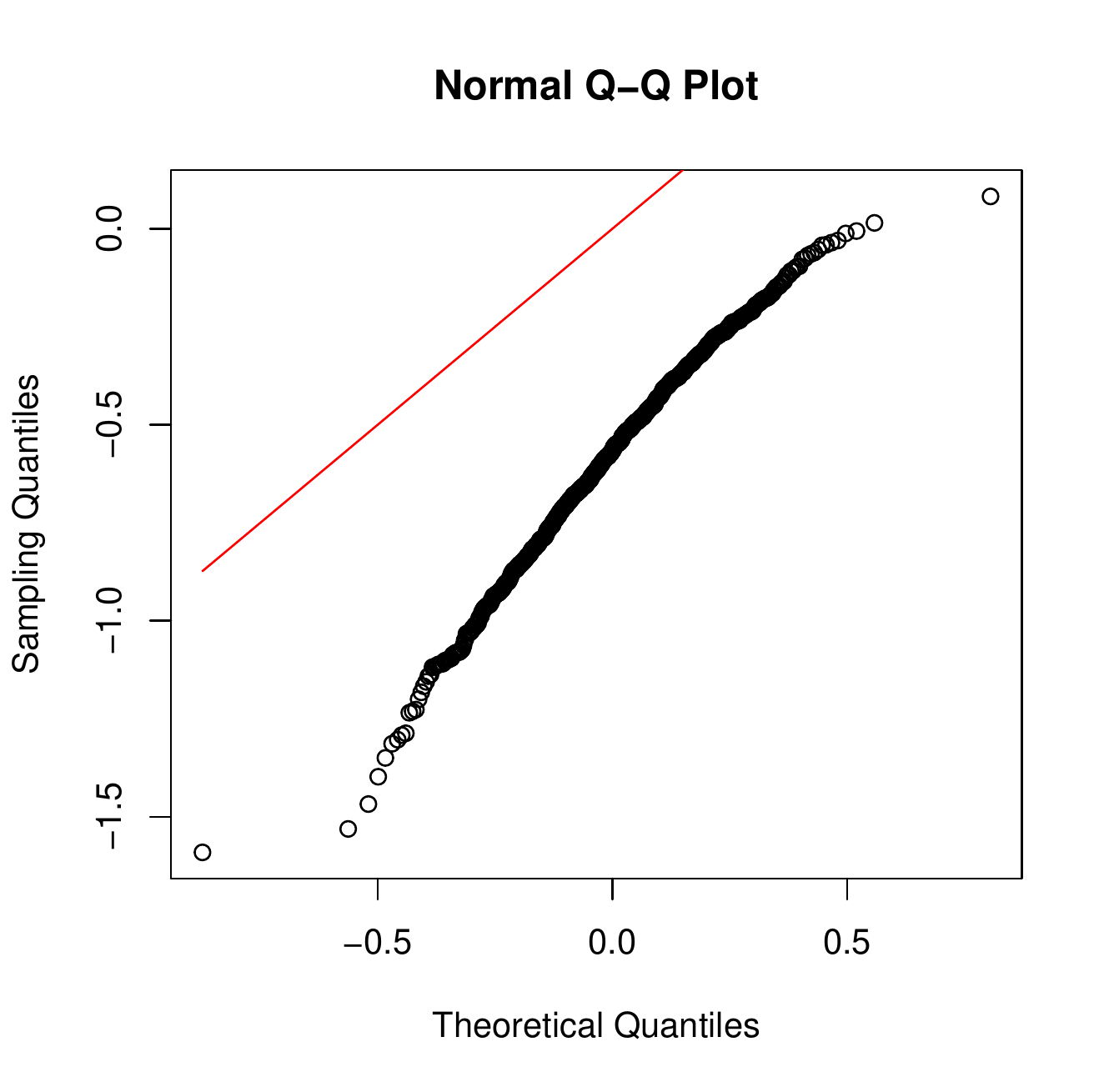}
	\includegraphics[width=.325\textwidth, height= .2\textheight]{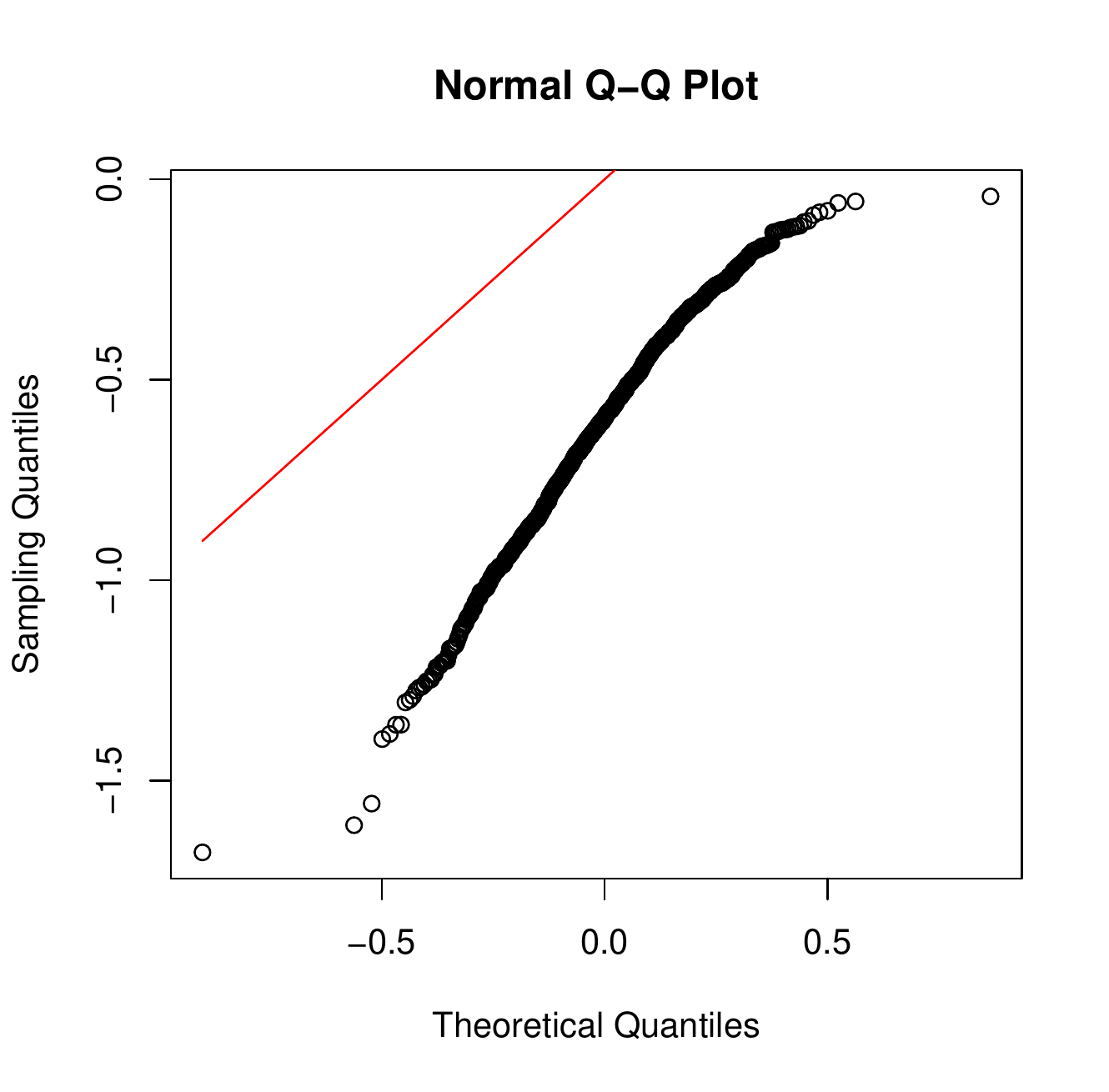}
	\caption{The same plots as in Figure~\ref{fig:Exp} for the case when $F$ is given in~\eqref{eq:F-PWA} and the functional of interest is $\tau(g,f) = \int x f^2(x) dx$ (i.e., $g(z,x) = x z^2$).}
	\label{fig:PWA}
\end{figure}

\section{Discussion}\label{sec:discussion}
In this paper we characterize the asymptotic  distribution of the nonparametric maximum likelihood (NPML) based plug-in estimator (see Theorem~\ref{thm:asymp_eff}) for estimating smooth integrated functionals of a monotone nonincreasing density $f$, under less restrictive assumptions on $f$. We also show that for a large class of functionals  the asymptotic limit is normal with mean zero and the semiparametric efficient variance (see Theorem~\ref{thm:asymp_eff_mu}), under minimal assumptions on $f$. In particular, we do not assume that the underlying true density $f$ is (i) smooth, or (ii) compactly supported, or (iii) lower bounded away from zero.


Finally, we believe that our basic proof idea can be used to study the asymptotic behavior of estimated functionals (based on the NPMLE) in other shape-restricted problems (e.g., decreasing convex densities or log-concave densities). A key open question in this direction is to express the NPMLE of the underlying density as a ``nice'' (e.g., Hadamard directionally differentiable) functional of the empirical distribution function of the data and to derive a corresponding weak convergence result analogous to~\citet[Theorem 2.1]{BF16}. 
As future research, we plan to explore this direction for other shape-restricted problems.

\section{Proofs}\label{section_appendix_proofs}
\subsection{Proof of Theorem~\ref{thm:asymp_eff}}\label{sec:Steps}
We first give an overview of the main ideas and steps involved in the proof of Theorem~\ref{thm:asymp_eff}.  
Observe that, 
\begin{eqnarray} \label{eq:D-1}
\tau(g,\hat f_n) - \tau(g,f) & =& \int_0^\infty \left\{g(\fhat(x),x)-g(f(x),x)\right\} dx \nonumber \\
&=& \int_0^\infty(\fhat(x)-f(x))\dot{g}(f(x),x)dx+w_n,
\end{eqnarray}
where 
$$w_n := \frac{1}{2} \int_0^\infty(\fhat(x)-f(x))^2 \, \ddot{g}(\xi_n(x),x)dx,$$
with $|\xi_n(x)-f(x)|\leq |\fhat(x)-f(x)|$, for every $x \ge 0$.

\noindent {\bf Step 1}: 
We claim that 
\begin{align}
n^{1/2} w_n = o_{\P}(1). \label{eqn:w_n}
\end{align}
We prove~\eqref{eqn:w_n} in Section~\ref{pf:MainT}. Below we complete the proof of Theorem \ref{thm:asymp_eff} assuming the validity of \eqref{eqn:w_n}. \newline

\noindent {\bf Step 2}: The first term in~\eqref{eq:D-1} can be handled as follows:
\begin{eqnarray*}
&& \int_0^\infty(\fhat(x)-f(x))\dot{g}(f(x),x)dx \\
& = & \left[(\hat F_n (x)-F(x))\dot{g}(f(x),x) \right]_0^\infty - \int_\s (\hat F_n (x)-F(x)) d[\dot{g}(f(x),x)]. 
\end{eqnarray*}
As $\dot{g}(f(\cdot),\cdot)$ is of finite total variation (by (A3)-(iii,b)) we have $(1-F(x))\dot{g}(f(x),x) \to 0$ as $x \to +\infty$. Thus,  
\begin{eqnarray*}
\sqrt{n} \int_0^\infty(\fhat(x)-f(x))\dot{g}(f(x),x)dx & = & - \sqrt{n} \int_\s (\hat F_n (x)-F(x)) \,d[\dot{g}(f(x),x)] \\
& = &- \int_\s \hat \D_n (x) \, d[\dot{g}(f(x), x)]  \\
& \stackrel{d}{\to} & - \int_\s \hat{\mathbb{G}}(x) \, d[\dot{g}(f(x), x)] 
\end{eqnarray*}
where the last step follows from~\citet[Theorem 2.1]{BF16} and a version of the continuous mapping theorem: Observe that (i) $ \hat \D_n \stackrel{d}{\to}  \hat{\mathbb{G}}$ in the space $\ell^{\infty}(\R_+)$; (ii) for any $\psi : \R_+ \to \R$ uniformly bounded, $\psi  \mapsto \int_\s \psi(x) d[\dot{g}(f(x),x)]$ is a continuous function. This proves that~\eqref{eq:Eff} holds.

\subsubsection{Proof of~\eqref{eqn:w_n}}\label{pf:MainT}
Fix $\epsilon>0$ and $\eta >0$. For $M^*>0$ a suitable constant (to be chosen later) we define $$A_n := \{ \|\fhat\|_{\infty} \le  M^* f(0+) \}, \qquad \mbox{for } \; n = 1,2,\ldots.$$		
	We have to study the probability
	\be
	\ &\P\left(n^{1/2} |w_n|>\eta\right) \le  \P\left(\sup\limits_{x \in \s}{|{\ddot{g}(\xi_n(x),x)}|}\|\fhat-f\|_2^2> 2 \eta \, n^{-1/2} \right) \nonumber \\
	&\leq \P\left(\sup\limits_{x \in \s, \; z \in [0,(M^*+1)f(0+)] }{|{\ddot{g}(z,x)}|}\|\fhat-f\|_2^2> 2 \eta\,  n^{-1/2}, \, A_n \right) +\P\left(A_n^c\right) \nonumber \\
	&\leq \P\left(\|\fhat-f\|_2^2> 2 K^{-1} \eta\,  n^{-1/2}, \, A_n\right) +\P\left(A_n^c\right),	\label{eq:w_n-bd} 
	\ee
	where we have used assumption (A3) which implies that $$\sup_{x \in \s, \; z \in [0,(M^*+1)f(0+)] }{|{\ddot{g}(z,x)}|} \le K,$$ and for any function $\psi:\R \to \R$, $\|\psi\|_2^2 := \int \psi^2(x) dx$.
	
	First, observe that by~\citet[Theorem 2 and Remark 3]{woodroofe1993penalized}, the sequence of random variables $\frac{\fhat(0+)}{f(0+)}$ is asymptotically $P$-tight and consequently, given any $\epsilon>0$ there exists $M^*>0$ (depending on $\epsilon>0$) such that for $n$ large enough (depending on $\epsilon>0$)
	\be 
	\P\left(\|\fhat\|_{\infty}> M^* f(0+)\right)\leq \frac{\epsilon}{2}. \label{eqn:linfty_bound}
	\ee
	Thus, the second term in~\eqref{eq:w_n-bd} is bounded from above by $\epsilon/2$.
	
	Next, note that $\|\fhat\|_{\infty}\leq M^*f(0+)$ implies that 
	\begin{eqnarray*}
		\|\fhat-f\|_2^2 & \leq & 2(\sqrt{M^*}+1)^2f(0+)\frac{1}{2}\int \left(\sqrt{\fhat(x)}-\sqrt{f(x)}\right)^2 dx \\
		& =& 2(\sqrt{M^*}+1)^2f(0+)\, H^2\left(\hat{P}_n,P\right),
	\end{eqnarray*}
	where $$H^2\left(\hat{P}_n,P\right) := \frac{1}{2}\int \left(\sqrt{\fhat(x)}-\sqrt{f(x)}\right)^2 dx $$ is the Hellinger divergence between $\hat{P}_n$ and $P$. As assumptions (A1)-(A2) are satisfied, by~\citet[Theorem 7.12]{geer2000empirical}, there exists  $c>0$ such that for all $\check C\geq c$, 
	\begin{equation*} 
	\P\left(H^2\big(\hat{P}_n,P\big)> \check C^2n^{-2/3}\right)\leq c\exp\left(-\frac{\check C^2n^{1/3}}{c^2}\right).\label{eqn:hellinger_bound}
	\end{equation*}
	Consequently, for an appropriate $C^*$ (depending on $M^*$) and $\tilde c$, we have  
	{\be 
		\P\left(\|\fhat-f\|_2^2>C^*n^{-2/3},\|\fhat\|_{\infty}\leq  M^* f(0+)\right)&\leq c\exp\left(-\tilde cn^{1/3}\right)\leq \frac{\epsilon}{2}, \qquad  \label{eqn:l2_bound}
		\ee}
	for $n$ large enough (depending on $\epsilon$). Consequently, using~\eqref{eqn:l2_bound} and~\eqref{eqn:linfty_bound}, for $n$  sufficiently large (which, in particular, implies that and $C^*n^{-2/3} \le 2 K^{-1} \eta \, n^{-1/2}$),~\eqref{eq:w_n-bd} yields $$\P\left(n^{1/2} |w_n|>\eta\right) \le \frac{\epsilon}{2} + \frac{\epsilon}{2}.$$ As $\epsilon$ and $\eta$ are arbitrary, we conclude $w_n=o_{\P}(n^{-1/2})$ as $n \to \infty$. \qed

\subsection{Proof of Theorem~\ref{thm:asymp_eff-2}}\label{sec:Thm-Normal-Limit}
We first note that $\sigma^2_{\mathrm{eff}}(g,f)<+\infty$. This follows from the fact that $\dot{g}^2(f(\cdot),\cdot)$ is a bounded function on $[0,\infty)$ (as $\dot{g}(f(\cdot),\cdot)$ is a function of bounded total variation on $[0,\infty)$).

Now, to show \eqref{eq:Eff-2}, first note that when $F$ is strictly concave on $\R_+$, it follows from Proposition~\ref{prop:BF} that, $\M_F' \xi = \xi$ for all $\xi \in C(\R_+)$. Thus, if $F$ is strictly concave then $\hat \G = \G$ and consequently,
\be
Y=-\int_\s \G(x)d \psi(x),
\ee
where $$\psi(x) :=\dot{g}(f(x),x) \qquad x \in \s.$$ Let $\s = [0,s]$ where $s \in \R$ or $s =+\infty$. {As $\psi(\cdot)$ is of finite total variation on $\s$ (by (A3)-(iii,b)), we shall assume without loss of generality that $\psi(s) := \lim_{x \to s} \psi(x) = 0$ (otherwise we can work with $\psi(\cdot) - \psi(s)$ instead of $\psi(\cdot)$ and that would not change the value of $Y$)}. 

If $s \in \R$ then it is easy to see that $Y$ is a centered Gaussian random variable. In the following we show that even when $s= +\infty$, $Y$ is a centered Gaussian random variable. We write $Y=-\int_{0}^N \G(x)d\psi(x) -\int_{N}^{\infty}\G(x)d\psi(x).$ Now $\E|\int_{N}^{\infty}\G(x)d\psi(x)|\leq \int_{N}^{\infty}\E|\G(x)|d\psi(x)$ which converges to $0$ (by (A3)-(iii,b)) as $N\rightarrow \infty$. Consequently $Y$ is the in probability (and hence in distribution) limit of $-\int_{0}^N \G(x)d\psi(x)$ as $N\rightarrow \infty$. Now for every $N$, $-\int_{0}^N \G(x)d\psi(x)$ is a Gaussian random variable. Therefore to show \eqref{eq:Eff-2}, we only need to show that the mean and variance of $-\int_{0}^N \G(x)d\psi(x)$ converge to $0$ and $\sigma^2_{\mathrm{eff}}(g,f)$ respectively as $N\rightarrow \infty$. Indeed by Fubini's Theorem $-\int_{0}^N \G(x)d\psi(x)$ has mean $0$ and this completes the mean part of the convergence. 

We next show that $\mathrm{Var}(Y)$ matches $\sigma^2_{\mathrm{eff}}(g,f)$. If $s = +\infty$, the same proof will also establish the validity of the convergence of the variance of $-\int_{0}^N \G(x)d\psi(x)$ to $\sigma^2_{\mathrm{eff}}(g,f)$ (by simply interpreting integrals appearing in the proof with respect to $\psi$ on $[0,\infty)$ as limits of integrals with respect to $\psi$ on $[0,N]$).

To operationalize our argument, we begin with the following simplification:
\be 
\int_0^\infty \G(x)d \psi (x) =\int_0^\infty  \mathbb{B}\circ F(x)d \psi(x)\stackrel{d}{=}\int_0^s \mathbb{W}\circ F(x)d \psi(x)-\mathbb{W}(1)\int_0^s F(x)d \psi(x),
\ee
where $\mathbb{W}$ is the standard Brownian motion on $[0,1]$ and $\stackrel{d}{=}$ stands for equality in distribution. To simplify notation, in the following every integral is taken on the set $[0,s]$. Therefore,
\be 
 \mathrm{Var}(Y)&=\mathrm{Var}\left(\int \mathbb{W}\circ F(x)d\psi(x)\right)+\mathrm{Var}\left(\mathbb{W}(1)\int F(x)d\psi(x)\right)\\&\qquad \qquad -2\,\mathrm{Cov}\left(\int \mathbb{W}\circ F(x)d \psi(x),\mathbb{W}(1)\int F(x)d\psi(x)\right)\\
&=\E\left(\int \mathbb{W}\circ F(x)d\psi(x)\right)^2+\left(\int F(x)d \psi(x)\right)^2\\&\qquad \qquad -2\int F(x)d \psi(x)\;\E\left(\mathbb{W}(1)\int \mathbb{W}\circ F(x)d\psi(x)\right).
\ee
Now, by simple integration by parts, since $\lim_{x \to \infty} \psi(x) = 0$,
\be 
\int F(x)d \psi(x)&=-\int  \psi(x)dF(x),
\ee
and consequently it is enough to prove the following lemma to complete the proof of Theorem~\ref{thm:asymp_eff-2} (as then $ \mathrm{Var}(Y)$ reduces to $ \mathrm{Var}(\psi(X)) =  \mathrm{Var}(\dot{g}(f(X),X))$).

\begin{lemma}\label{lemma:Ito_Isometry}
	The following identities hold under assumptions (A1)-(A3):
	\be 
	\E\left(\mathbb{W}(1)\int \mathbb{W}\circ F(x)d\psi(x)\right)&=\int F(x)d \psi(x), \quad \text{and}\\
	\E\left(\int \mathbb{W}\circ F(x)d \psi(x)\right)^2&=\int \psi^2(x)dF(x).
	\ee
\end{lemma}
\begin{proof}
	The first identity is obvious since under our assumptions one can interchange the integral and expectation (in particular Fubini's Theorem is applicable since $\int_{0}^{s}\E(|\G(x)|)d(\dot{g}(f(x),x))<+\infty$ by part (iii,b) of (A3)) and observe that $$\E(\mathbb{W}(1)\mathbb{W}(F(x))) = \mathrm{Cov}\left(\mathbb{W}(1),\mathbb{W}(F(x))\right)=\min\{1,F(x)\}=F(x).$$
	
	For the second identity (once again we freely interchange integrals and expectation by Fubini's Theorem which is applicable by part (iii,b) of (A3)), note that 
	\be 
	\E\left(\int_0^s \mathbb{W}\circ F(x)d \psi(x)\right)^2&=\E\left(\int_0^s\int_0^s \mathbb{W}\circ F(x)\mathbb{W}\circ F(x')d \psi(x)d \psi(x')\right)\\
	&=2\int_{0}^{s}\int_{0}^{x' \wedge s}F(x)d \psi(x)d \psi(x').\label{eqn:eff_main_term}
	\ee
	Now, let $v(x') :=\int_{0}^{x'}F(x)d \psi(x)$, for $x' \ge 0$. Then by repeated integration by parts and noting that $\psi(+\infty) = 0$ by our assumption, 
	\be 
	\ & \int_{0}^{s}\int_{0}^{x' \wedge s}F(x)d \psi(x)d \psi(x')\\
	&=\lim_{x'\rightarrow s}v(x') \psi(x')-v(0) \psi(0)-\int_{0}^{s} \psi(x')dv(x')\\
	&= -\int_{0}^{s} \psi(x')dv(x')\\
	&=-\left[\int_{0}^{s} \psi(x')F(x')d \psi(x')\right]\\
	&=-\left[\lim_{x'\rightarrow s}F(x')\frac{1}{2}\psi^2(x')-F(0)\frac{1}{2} \psi^2(0)-\frac{1}{2}\int_{0}^{s}\psi^2(x')dF(x')\right]\\
	&=\frac{1}{2}\int_{0}^{s}\psi^2(x')dF(x').\label{eqn:eff_main_term_simplify}
	\ee
	Plugging back \eqref{eqn:eff_main_term_simplify} into \eqref{eqn:eff_main_term} completes the proof of Lemma \ref{lemma:Ito_Isometry}.
\end{proof}

\subsection{Proof of Theorem~\ref{thm:asymp_eff_mu}}\label{sec:Reduction}
Proposition~\ref{prop:BF} will help us demonstrate that indeed $Y$ (as defined in~\eqref{eq:Eff_nu}) has a Gaussian distribution with the desired efficient variance. However, to give more intuition to the reader, we begin by discussing two simple scenarios.
	
{\bf Case (i)}: First, let $F$ be strictly concave on $\R_+$. From Proposition~\ref{prop:BF} it follows that $\M_\theta'$ is linear if and only if $\theta$ is strictly concave on $\R_+$ (\cite{BF16}). Thus, $\M_F' \xi = \xi$ for all $\xi \in C(\R_+)$. Thus, if $F$ is strictly concave then $\hat \G = \G$ and consequently $$Y {=} - \int \G(x)\, d [h'(f(x))].$$ In this form, it is easy to see that $Y$ is a centered Gaussian random variable. The fact that its variance matches the efficiency bound can be demonstrated as in the proof of Theorem~\ref{thm:asymp_eff-2}  in Section~\ref{sec:Thm-Normal-Limit} (with $\psi(x) = h'(f(x))$).

{\bf Case (ii)}: Next consider the other extreme case, i.e., $F$ is piecewise affine. In this case $f$ is piecewise constant with jumps (say) at $0<t_1<t_2<\ldots< t_k<\infty$ and values $v_1>\ldots> v_k$, for some $k \ge 2$ integer, i.e., $$f(x) = \sum_{i=1}^k v_i \mathbf{1}_{(t_{i-1},t_i]}(x),$$ where $t_0 \equiv 0$. As $f$ only takes $k+1$ values (namely, $\{v_1, \ldots, v_k, 0\}$) we have that $$Y =  - \sum_{i=1}^k a_i \hat \G(t_i)$$ where $$a_i  := \{h'(f(t_i-))  - h'(f(t_i+))\}.$$ Further, from Proposition~\ref{prop:BF}, we have, for every $i = 1,\ldots, k$, $\hat \G(t_i) \equiv \M_F'\G(t_i) = \G(t_i).$ As $\G(t_k) = 0$ a.s., we have $$Y = \sum_{i=1}^{k-1}  a_i  \G(t_i)=\int \G(x)\,d[h'(f(x))]$$ which is the same distribution as in the case when $F$ was strictly concave.

Since in both the above cases we have $Y=\int \G(x)d [h'(f(x))]$, one can conjecture that the result should be generally true for any concave $F$. This intuition indeed turns out to be correct as shown below. \newline

Let $\psi(x) := h'(x)$, for $ x \ge 0$. Note that by assumption (A3) applied to $g(z,x)=h(z)$, we have that $\psi(f(x))$ has bounded total variation. Consequently, let us assume that the function $\psi \circ f$ is monotone (otherwise we can split $\psi \circ f$ as the difference of two monotone functions and apply the following to each part). Note that it is enough to show that for any $\xi \in C(\R_+)$,  $$\int \M_F' \xi(x) d[\psi(f(x))] = \int \xi(x) d[\psi(f(x))]. $$ Then we can take $\xi$ to be $\G \equiv \B \circ F$ (which is almost surely continuous) and we will obtain the desired result.
	
Let $$I := \{x \in (0,\infty): F \;\mbox{is affine in an open interval around $x$}\}.$$  It is easy to show that $I$ is a Borel measurable set. Consider the collection $\{T_{F,x}: x \in I\}$, where $T_{F,x}$ is defined in Proposition~\ref{prop:BF} (also see~\citet[Proposition 2.1]{BF16}). First note that $x \in T_{F,x}$, for every $x \in I$. We now claim that there are at most countably many such distinct $T_{F,x}$'s, as $x$ varies in $I$. This follows from the fact that for any $x$ in the support of $F$, $F$ has a strictly positive slope on $T_{F,x}$ (i.e., $F$ cannot be affine with slope 0 on $T_{F,x}$, as $F$ is concave and nondecreasing and $x$ is in the support of $F$) which implies that we can associate a unique rational number in the range of $F$ to this interval $T_{F,x}$. Let $\{T_{F,x_i}\}_{i\ge 1}$ (for $x_i \in (0,\infty)$) be an enumeration of this countable collection $\{T_{F,x}: x \in I\}$. Obviously, $I \subset \cup_{i=1}^\infty T_{F,x_i}$.
	
	Let $\xi \in C(\R_+)$.  Then, 
	\begin{eqnarray*}
		&&	\left|\int \M_F' \xi(x) d[\psi(f(x))] - \int \xi(x) d[\psi(f(x))] \right| \\ 
		& \le & \int_I |\M_F' \xi(x) -\xi(x)| d[\psi(f(x))] + \int_{[0,\infty)\setminus I} |\M_F' \xi(x) - \xi(x)| d[\psi(f(x))] \\
		& \le & \sum_{i=1}^\infty \int_{T_{F,x_i}} |\M_{T_{F,x_i}} \xi(x) -\xi(x)| d[\psi(f(x))] 
	\end{eqnarray*}
	where the last inequality follows from the fact that for $x \in [0,\infty)\setminus I$, $\M_F' \xi(x) = \xi(x)$ (see \citet[Remark 2.2]{BF16}) and the above integrals are viewed as a Lebesgue-Stieltjes. Notice now that on the open interval $T_{F,x_i}$, $f$ is constant (as $F$ is affine on $T_{F,x_i}$). Hence each of the integrals in the last display equals 0. This completes the proof. \qed
	
\subsection{Proof of Lemma \ref{lem:Basic}}
	Let $X_{(1)} < X_{(2)} <\ldots < X_{(n)}$ be the ordered data points and let $$\hat \theta_i = \hat f_n(X_{(i)}), \qquad \mbox{ for } i=1,\ldots, n,$$ be the fitted values of the Grenander estimator at the ordered data points. A simple characterization of $\hat \theta := (\hat \theta_1,\ldots, \hat \theta_n)$ is given by $$\hat \theta = \argmin_{\theta \in \C} \sum_{i=1}^n \left(\frac{1}{n(X_{(i)} - X_{(i-1)})} - \theta_i\right)^2 (X_{(i)} - X_{(i-1)}),$$ where $$ \C = \{\theta \in \R^n: \theta_1 \ge \cdots \ge \theta_n\},$$ and $X_{(0)} \equiv 0$; see e.g.,~\citet[Exercise 2.5]{GJ14}. From the characterization of projection onto the  closed convex cone $\C$ it follows that for any function $h: \R \to \R$ (see e.g.,~\citet[Theorem 1.3.6]{RWD88})
	\begin{eqnarray*}
		&& \sum_{i=1}^n \left(\frac{1}{n(X_{(i)} - X_{(i-1)})} - \hat \theta_i\right) (X_{(i)} - X_{(i-1)}) h(\hat \theta_i) =0 \\
		& \Leftrightarrow & \frac{1}{n} \sum_{i=1}^n h(\hat \theta_i) = \sum_{i=1}^n h(\hat \theta_i) \hat \theta_i  (X_{(i)} - X_{(i-1)}) \\
		& \Leftrightarrow & \P_n[h\circ \hat f_n] = \hat P_n[h\circ \hat f_n]. \end{eqnarray*}
	The last step follows from the fact that $$\hat P_n[h \circ \hat f_n]  = \int h( \hat f_n(x)) \hat f_n(x) dx = \sum_{i=1}^n h (\hat \theta_i) \hat \theta_i (X_{(i)} - X_{(i-1)}),$$ as $ \hat f_n(\cdot)$ is constant on the interval $(X_{(i-1)}, X_{(i)}]$, for every $i=1,\ldots, n$. \qed

\subsection{Proof of Theorem \ref{theorem:uniform}}\label{pf:Thm-Unif}
First let us define $g:[0,\infty) \to \R$ as   $g(x) :=h(x)x$, for $x \ge 0$, and consequently $g\in C^4([0,\infty))$, $g''(x)=h''(x)x+2h'(x)$, and 

$$\mu(h,\fhat)-\mu(h,f)=\int_0^1 [h(\fhat(x))-h(f(x))]dx.$$

For $x \in [0,1]$, note that for some $\xi_n(x)$ lying between $f(x) \equiv 1$ and $\fhat(x)$ one has by a four term Taylor expansion (allowed by the assumption that $h\in C^4([0,\infty))$)
\be 
\ & \frac{n\int_0^1 [h(\fhat(x))-h(1)]dx-\frac{1}{2}h''(1)\log{n}}{\sqrt{3[\frac{1}{2}h''(1)]^2\log{n}}}\\&=\frac{n\int_0^1 (\fhat(x)-1)^2dx-\log{n}}{\sqrt{3\log{n}}} \label{eq:Display}\\
&\quad +\frac{1}{6}\frac{n\int_0^1 h^{(3)}(1)(\fhat(x)-1)^3 dx}{\sqrt{{3[\frac{1}{2}h''(1)]^2\log{n}}}}+\frac{1}{24}\frac{n\int_0^1 h^{(4)}(\xi_n(x))(\fhat(x)-1)^4 dx}{\sqrt{{3[\frac{1}{2}h''(1)]^2\log{n}}}}.
\ee
Therefore, if we prove that 
$$\frac{n\int_0^1 (\fhat(x)-1)^k dx}{\sqrt{\log{n}}}=o_{\P}(1),\quad \mbox{for } \, k=3,4,$$ then  
\be 
\left|\frac{n\int_0^1 h^{(4)}(\xi_n(x))(\fhat(x)-1)^4 dx}{\sqrt{3\log{n}}} \right|&\leq \sup\limits_{z\in [0,\|\fhat\|_{\infty} \vee 1]}|h^{(4)}(z)|\frac{n\int_0^1 (\fhat(x)-1)^4 dx}{\sqrt{3\log{n}}} =o_{\P}(1),
\ee
since $\|\fhat\|_{\infty}=O_{\P}(1)$ by \cite{woodroofe1993penalized} and $h^{(4)}$ is bounded on compact intervals. Therefore, modulo the claim above, one readily obtains the desired result~\eqref{eq:Unif-Lim} from~\eqref{eq:Display} and~\eqref{eq:Unif-2}.

The next two subsections are devoted to understanding $\frac{n\int_0^1 (\fhat(x)-1)^kdx}{\sqrt{\log{n}}}$ for $k=3,4$. In the analysis, we crucially use \citet[Theorem 2.1]{GP83} and consequently need the following notation to proceed. Let $$T_n :=\sum_{j=1}^n jN_j, \quad  S_n:=\sum_{j=1}^n\sum_{i=1}^{N_j} S_{ji},$$ where $\{S_{ji}: i\geq 1, j\geq 0\}$ are independent random variables with $S_{ji}\sim \mathrm{Gamma}(j,1)$, for $j\geq 1$, and $S_{0i}\sim \mathrm{Gamma}(1,1)$, and $N_j\sim \mathrm{Poisson}(1/j)$ for $j\geq 1$ and $N_0=1$ a.s.; further, $\{N_j\}$ and $\{S_{ji}\}$ are independent. Also let, 
	\be 
	W_n=\frac{1}{n}\sum_{j=1}^n jN_j,\quad V_n=n^{-1/2}\sum_{j=1}^n\sum_{i=1}^{N_j}(S_{ji}-j).
	\ee
With this notation we are ready to analyze $\frac{n\int_0^1 (\fhat(x)-1)^kdx}{\sqrt{\log{n}}}$ for $k=3,4$.

\subsubsection{Proof of $\frac{n\int_0^1 (\fhat(x)-1)^3dx}{\sqrt{\log{n}}}=o_{\P}(1)$}
Let us start by considering the distribution of 
 	\be 
 	L_n :=\int_0^1 (\fhat(x)-1)^3dx=\int \fhat^3(x)dx-3\int \fhat^2(x)dx+2.
 	\ee
By an argument similar to that in~\citet[Section 3]{GP83}, it can be seen that $L_n$ has the same distribution as that of
\be 
L_n^{*}& :=\frac{1}{n}\sum_{j=1}^n j^3\sum_{i=1}^{N_j}\frac{1}{S_{ji}^2} 
-\frac{3}{n}\sum_{j=1}^n j^2\sum_{i=1}^{N_j}\frac{1}{S_{ji}}+2 \, \Big| \,\{T_n=n, S_n=n\}.
\ee

For constants $c_0,c_n,\gamma_n$, we define the sequence of random variables 
\be 
U_n(c_0,c_n,\gamma_n)  := & \frac{1}{c_n}\left[\sum_{j=1}^n j^3\sum_{i=1}^{N_j}\left(\frac{1}{S_{ji}^2}-\frac{1}{j^2}\right)-3\sum_{j=1}^n j^2\sum_{i=1}^{N_j}\left(\frac{1}{S_{ji}}-\frac{1}{j}\right) \right.\\
	& +\; \left.c_0\sum_{j=1}^n\sum_{i=1}^{N_j}(S_{ji}-j)-\gamma_n\right].\qquad \qquad 
\ee
{The particular choice of $c_0,c_n,\gamma_n$ is crucial in the analysis that follows. In particular, a specific choice for the constant $c_0$ implies certain fine-tuned cancellations which are necessary for controlling the asymptotic variance of $U_n$ at a desired level. It is worth noting that such a definition and the eventual choice of $c_0=1$ is also present in \citet[Equation (3.2)]{GP83}. Our choice of $c_0,c_n,\gamma_n$ is more tailored to the current problem.}
Finally, conditional on the joint event that $W_n=1$ and $V_n=0$ (i.e., $T_n=n$ and $S_n=n$) one has 
 \begin{equation}\label{eq:U_n-L_n}
 U_n(c_0,c_n,\gamma_n) \big| \{W_n=1, V_n=0\}\stackrel{d}{=}c_n^{-1}(nL_n^*-\gamma_n)\big|\{T_n=n,S_n=n\}.
 \end{equation}
The following lemma, crucial to our proof, is the analogue of \citet[Lemma 3.1]{GP83}; a correction to the characteristic function formula in \citet[Lemma 3.1]{GP83} was recently pointed out in \cite{piet_manuscript}.
  \begin{lemma}\label{lemma:asymp_joint_distr_uvw_4}
  	Let $c_n=\sqrt{\log{n}}$, $\gamma_n=0$ and $c_0=-1$. Then $(U_n,V_n,W_n)\stackrel{d}{\rightarrow} (U,V,W)$ where $U=\delta_0$, the dirac measure at 0, and $(V,W)$ has the infinitely divisible characteristic function $\phi_{V,W}(s,t)=\exp\left(\int_0^1(e^{itu-s^2u/2}-1)u^{-1}du\right)$.
  \end{lemma}
  \begin{proof}
  
  	With the choice of $c_0,c_n,\gamma_n$ provided in the statement of the theorem, denote $U_n \equiv U_n(c_0,c_n,\gamma_n)$. We only prove the fact that $U_n \stackrel{\P}{\rightarrow} 0$. The subsequent conclusion follows along similar lines as in the proof of \citet[Lemma 3.1]{GP83}. In the following analysis we repeatedly use the fact that $\E\left(S_{ji}^{-k}\right)=\frac{\Gamma(j-k)}{\Gamma(j)}$ for $j>k$. 
  	
  	First, define $U_n^*$ to be $U_n$ without its first $3$ summands in $j$, i.e., 
{\small $$ U_n^* := \frac{1}{c_n}\left[\sum_{j=4}^n j^3\sum_{i=1}^{N_j}\left(\frac{1}{S_{ji}^2}-\frac{1}{j^2}\right)-3\sum_{j=4}^n j^2\sum_{i=1}^{N_j}\left(\frac{1}{S_{ji}}-\frac{1}{j}\right)-\sum_{j=4}^n\sum_{i=1}^{N_j}(S_{ji}-j) \right].$$}
We will show that $U_n^*\stackrel{\P}{\rightarrow} 0$ and as $U_n-U_n^*\stackrel{\P}{\rightarrow} 0$ (since $c_n\rightarrow \infty$), this will establish the desired result. Note that, 
  \be 
 c_n\E\left(U_n^*\right)
  &=\sum_{j=4}^n j^3\E(N_j)\left[\frac{\Gamma(j-2)}{\Gamma(j)}-\frac{1}{j^2}\right]-3\sum_{j=4}^n j^2\E(N_j)\left[\frac{\Gamma(j-1)}{\Gamma(j)}-\frac{1}{j}\right]\\
  &=\sum_{j=4}^n \frac{4}{(j-1)(j-2)}.
  \ee
As $\limsup_{n\rightarrow \infty}|\sum_{j=4}^n \frac{4}{(j-1)(j-2)}|<\infty$, one has $\E(U_n^*)\rightarrow 0$ (as $c_n\rightarrow \infty$). Subsequently it is enough to prove that $\mathrm{Var}(U_n^*)\rightarrow 0$ which in turn is implied by proving that $\limsup_{n\rightarrow \infty}\mathrm{Var}(c_n U_n^*) <\infty$. First note that 
  \be 
  \mathrm{Var}(\E(c_n U_n^*|N_1,\ldots,N_n))&=\sum_{j=4}^n \left(\frac{4j}{(j-1)(j-2)}\right)^2\mathrm{Var}(N_j)\\
  &=\sum_{j=4}^n \left(\frac{4j}{(j-1)(j-2)}\right)^2\frac{1}{j},
  \ee
  and consequently, $\limsup_{n\rightarrow \infty}\mathrm{Var}(\E(c_n U_n^*|N_1,\ldots,N_n)) <\infty$. Next, note that 
  \begin{eqnarray*}
\mathrm{Var}(c_nU_n^*|N_1,\ldots,N_n) 
  &=& \sum_{j=4}^n N_j\left[ j^6\mathrm{Var}(\frac{1}{S_{j1}^2})+9j^4\mathrm{Var}(\frac{1}{S_{j1}})+\mathrm{Var}(S_{j1}) \right. \\
  && \left. -6j^5\mathrm{Cov}(\frac{1}{S_{j1}^2},\frac{1}{S_{j1}})-2j^3\mathrm{Cov}(\frac{1}{S_{j1}^2},S_{j1})+6\mathrm{Cov}(\frac{1}{S_{j1}},S_{j1})\right]\\
  &=& \sum_{j=4}^n N_j\frac{15 j^5 + 129 j^4 - 404 j^3 + 116 j^2 + 48 j}{(j-1)^2(j-2)^2(j-3)(j-4)}.
  \end{eqnarray*}
  Consequently, 
  \be 
  \ & \limsup_{n\rightarrow \infty}\E(\mathrm{Var}(c_nU_n^*|N_1,\ldots,N_n))\\
  &=\limsup_{n\rightarrow \infty}\sum_{j=4}^n \frac{15 j^4 + 129 j^3 - 404 j^2 + 116 j + 48 }{(j-1)^2(j-2)^2(j-3)(j-4)}<\infty.
  \ee
  This completes the proof of the lemma. 
\end{proof}
 
  We finally note that Lemma \ref{lemma:asymp_joint_distr_uvw_4} immediately implies that  $U_n$ is asymptotically independent of $(W_n,V_n)$ and that $U_n-U_m \stackrel{\P}{\rightarrow} 0$ for any $n,m\rightarrow \infty$. Consequently, the analysis provided in the proof of  \citet[Theorem 3.1]{GP83} goes through verbatim yielding that $$U_n|\{W_n=1, V_n=0\}\stackrel{d}{\rightarrow}0$$ which in turn proves that (using~\eqref{eq:U_n-L_n}) $c_n^{-1}(nL_n^*-\gamma_n)\big| \{T_n=n,S_n=n\} \stackrel{d}{\rightarrow}0$ for our choice of $\gamma_n=0$. 
 This yields the desired result.
 
 \subsubsection{Proof of $\frac{n\int_0^1 (\fhat(x)-1)^4 dx}{\sqrt{\log{n}}}=o_{\P}(1)$}

The proof technique is similar to that of the case $\frac{\int_0^1 (\fhat(x)-1)^3 dx}{\sqrt{\log{n}}}$ albeit with much more cumbersome details and a different choice of $c_0$. As before, let us start by considering the distribution of 
 	\be 
 	L_n& :=\int_0^1 (\fhat(x)-1)^4dx=\int \fhat^4(x)dx-4\int \fhat^3(x)dx+6\int \fhat^2(x)dx-3.
 	\ee
As argued in \citet[Section 3]{GP83}, it is equivalent to studying the (conditional) distribution of 
\be 
L_n^*:=\frac{1}{n}\sum_{j=1}^n j^4\sum_{i=1}^{N_j}\frac{1}{S_{ji}^3} 
-\frac{4}{n}\sum_{j=1}^n j^3\sum_{i=1}^{N_j}\frac{1}{S_{ji}^2}+\frac{6}{n}\sum_{j=1}^n j^2\sum_{i=1}^{N_j}\frac{1}{S_{ji}}-3 \Big| \, \{ T_n=n, S_n=n\}.
\ee
We once again introduce a sequence of random variables 
\be 
U_n(c_0,c_n,\gamma_n) &:= \frac{1}{c_n}\left[ \sum_{j=1}^n j^4\sum_{i=1}^{N_j}\left(\frac{1}{S_{ji}^3}-\frac{1}{j^3}\right)-4\sum_{j=1}^n j^3\sum_{i=1}^{N_j}\left(\frac{1}{S_{ji}^2}-\frac{1}{j^2}\right)\right.\\
 & + \left.6\sum_{j=1}^n j^2\sum_{i=1}^{N_j}\left(\frac{1}{S_{ji}}-\frac{1}{j}\right) 
	+c_0\sum_{j=1}^n\sum_{i=1}^{N_j}(S_{ji}-j)-\gamma_n\right].
\ee
Then, conditional on the joint event that $W_n=1$ and $V_n=0$ (i.e., $T_n=n$ and $S_n=n$) one has  
 \be 
 U_n(c_0,c_n,\gamma_n) \big| \{W_n=1, V_n=0\}\stackrel{d}{=}c_n^{-1}(nL_n^*-\gamma_n)|\{T_n=n,S_n=n\}.
 \ee
 
Once again we have a crucial lemma which is the analogue of \citet[Lemma 3.1]{GP83}; a correction to the characteristic function formula in \citet[Lemma 3.1]{GP83} was recently pointed out in \cite{piet_manuscript}. 
 
 \begin{lemma}
 \label{lemma:asymp_joint_distr_uvw_5}
 	Let $c_n=\sqrt{\log{n}}$, $\gamma_n=0$ and $c_0=1$. Then $(U_n,V_n,W_n)\stackrel{d}{\rightarrow} (U,V,W)$ where $U=\delta_0$ and $(V,W)$ has the infinitely divisible characteristic function $\phi_{V,W}(s,t)=\exp\left(\int_0^1(e^{itu-s^2u/2}-1)u^{-1}du\right)$.
 \end{lemma}
 \begin{proof}
With the above choice of $c_0,c_n,\gamma_n$, denote by $U_n \equiv U_n(c_0,c_n,\gamma_n)$. We only prove the fact that $U_n \stackrel{\P}{\rightarrow} 0$. The subsequent conclusion follows along similar lines as in the proof of \citet[Lemma 3.1]{GP83} and Lemma~\ref{lemma:asymp_joint_distr_uvw_4}. 
 	
 	First, define $U_n^*$ to be $U_n$ without it's first $4$ summands in $j$, i.e.,
 \be 
 U_n^* & =  c_n^{-1}\left[  \sum_{j=5}^n j^4\sum_{i=1}^{N_j}\left(\frac{1}{S_{ji}^3}-\frac{1}{j^3}\right)-4\sum_{j=5}^n j^3\sum_{i=1}^{N_j}\left(\frac{1}{S_{ji}^2}-\frac{1}{j^2}\right) \right.\\
 & \quad \left. +6\sum_{j=5}^n j^2\sum_{i=1}^{N_j}\left(\frac{1}{S_{ji}}-\frac{1}{j}\right) 
 	+\sum_{j=5}^n\sum_{i=1}^{N_j}(S_{ji}-j)\right].
 \ee
We show that $U_n^*\stackrel{\P}{\rightarrow} 0$ and indeed by $U_n-U_n^*\stackrel{\P}{\rightarrow} 0$ (since $c_n\rightarrow \infty$) we establish the desired claim. Note that, 
 \be 
c_n\E\left(U_n^*\right)
 &=\sum_{j=5}^n j^4\E(N_j)\left[\frac{\Gamma(j-3)}{\Gamma(j)}-\frac{1}{j^3}\right]-4\sum_{j=5}^n j^3\E(N_j)\left[\frac{\Gamma(j-2)}{\Gamma(j)}-\frac{1}{j^2}\right]\\
 &+6\sum_{j=5}^n j^2\E(N_j)\left[\frac{\Gamma(j-1)}{\Gamma(j)}-\frac{1}{j}\right]\\
 &=\sum_{j=5}^n \frac{2(j-9)}{(j-1)(j-2)(j-3)}.
 \ee
 Now $\limsup_{n\rightarrow \infty}|\sum_{j=5}^n \frac{2(j-9)}{(j-1)(j-2)(j-3)}|<\infty$ by standard ratio test. Consequently, using $c_n\rightarrow \infty$ one has $\E(U_n^*)\rightarrow 0$.
 	
 Subsequently it is enough to prove that $\mathrm{Var}(U_n^*)\rightarrow 0$ which in turn is implied by proving that $\limsup_{n\rightarrow \infty}\mathrm{Var}(c_n U_n^*) <\infty$. First note that 
 \be 
 \mathrm{Var}(\E(c_n U_n^*|N_1,\ldots,N_n))&=\sum_{j=5}^n \left(\frac{2j(j-9)}{(j-1)(j-2)(j-3)}\right)^2\mathrm{Var}(N_j)\\
 &=\sum_{j=5}^n \left(\frac{2j(j-9)}{(j-1)(j-2)(j-3)}\right)^2\frac{1}{j},
 \ee
 and consequently, $\limsup_{n\rightarrow \infty}\mathrm{Var}(\E(c_n U_n^*|N_1,\ldots,N_n)) <\infty$. Next note that 
 \be 
 \ & \mathrm{Var}(c_nU_n^*|N_1,\ldots,N_n)\\
 &=\sum_{j=5}^n N_j\left[
 j^8\mathrm{Var}(\frac{1}{S_{j1}^3})+16j^6\mathrm{Var}(\frac{1}{S_{j1}^2})+36j^4\mathrm{Var}(\frac{1}{S_{j1}})+\mathrm{Var}(S_{j1}) \right.\\
 & \qquad \qquad -8j^7\mathrm{Cov}(\frac{1}{S_{j1}^3},\frac{1}{S_{j1}^2})+12j^6\mathrm{Cov}(\frac{1}{S_{j1}^3},\frac{1}{S_{j1}})+2j^4\mathrm{Cov}(\frac{1}{S_{j1}^3},S_{j1})\\
&  \qquad \qquad -\left.48j^5\mathrm{Cov}(\frac{1}{S_{j1}^2},\frac{1}{S_{j1}})-8j^3\mathrm{Cov}(\frac{1}{S_{j1}^2},S_{j1})+12j^2\mathrm{Cov}(\frac{1}{S_{j1}},S_{j1})\right]\\
 &=\sum_{j=5}^n N_j\frac{96 j^7 + 4297 j^6 - 14259 j^5 - 28198 j^4 + 102180 j^3 - 33336 j^2 - 4320 j
 }{(j-1)^2(j-2)^2(j-3)^2(j-4)(j-5)(j-6)}.
 \ee
 Consequently, 
{\small  \be 
 \ & \limsup_{n\rightarrow \infty}\E(\mathrm{Var}(c_nU_n^*|N_1,\ldots,N_n))\\
 &=\limsup_{n\rightarrow \infty}\sum_{j=5}^n \frac{96 j^6 + 4297 j^5 - 14259 j^4 - 28198 j^3 + 102180 j^2 - 33336 j - 4320 
  }{(j-1)^2(j-2)^2(j-3)^2(j-4)(j-5)(j-6)} <\infty.
 \ee}
 This completes the proof of the lemma. 
 \end{proof}
 
 We finally note that as before Lemma \ref{lemma:asymp_joint_distr_uvw_4} immediately implies that in the asymptotic distributional limit, $U_n$ is independent of $W_n,V_n$ and that $U_n-U_m\stackrel{\P}{\rightarrow} 0$ for any $n,m\rightarrow \infty$. Consequently, the analysis provided in proving \citet[Theorem 3.1]{GP83} goes through verbatim yielding that $U_n|\{W_n=1, V_n=0\}\stackrel{d}{\rightarrow}0$ which in turn proves that $$c_n^{-1}(nL_n^*-\gamma_n)|\{T_n=n,S_n=n\}\stackrel{d}{\rightarrow}0$$ for our choice of $\gamma_n=0$. Consequently, $\frac{\int_0^1 (\fhat(x)-1)^4 dx}{\sqrt{\log{n}}}\stackrel{\P}{\rightarrow}0.$

\section*{Acknowledgements}
The authors would like to thank Jon Wellner for many insightful and helpful comments that led to major improvements in some of the results in this manuscript. The authors would also like to thank Piet Groeneboom for helpful conversations and for sharing his working manuscript.

\bibliographystyle{chicago}
\bibliography{AG}

\end{document}